\documentclass[11pt, reqno]{amsart}
\usepackage{amssymb}
\usepackage{eucal}
\usepackage{amsmath}
\usepackage{amscd}
\usepackage[dvips]{color}
\usepackage{multicol}
\usepackage[all]{xy}           %xypic macro for latex2.09
\usepackage{graphicx}
\usepackage{color, ulem}
\usepackage{colordvi}
\usepackage{xspace}
\usepackage{tikz}
\usepackage{ifpdf}
\ifpdf
 \usepackage[colorlinks,final,backref=page,hyperindex]{hyperref}
\else
 \usepackage[colorlinks,final,backref=page,hyperindex,hypertex]{hyperref}
\fi

\begin{document}
\newcommand {\emptycomment}[1]{} %to remove paragraphs

\baselineskip=14pt
\newcommand{\nc}{\newcommand}
\newcommand{\delete}[1]{}
\nc{\mfootnote}[1]{\footnote{#1}} % Use this to show footnotes
\nc{\todo}[1]{\tred{To do:} #1}

\delete{
\nc{\mlabel}[1]{\label{#1}}  % Use this to suppress names
\nc{\mcite}[1]{\cite{#1}}  % Use this to suppress names
\nc{\mref}[1]{\ref{#1}}  % Use this to suppress names
\nc{\meqref}[1]{\ref{#1}} % Use this to suppress names
\nc{\mbibitem}[1]{\bibitem{#1}} % Use this to show number
}

%\delete{
\nc{\mlabel}[1]{\label{#1}  % Use the next two lines to show names
{\hfill \hspace{1cm}{\bf{{\ }\hfill(#1)}}}}
\nc{\mcite}[1]{\cite{#1}{{\bf{{\ }(#1)}}}}  % Use this lines to show names
\nc{\mref}[1]{\ref{#1}{{\bf{{\ }(#1)}}}}  % Use this lines to show names
\nc{\meqref}[1]{\eqref{#1}{{\bf{{\ }(#1)}}}} % Use this lines to show names
\nc{\mbibitem}[1]{\bibitem[\bf #1]{#1}} % Use this to show name
%}

%%%%%%%%%%%%%%%%%%%%%%%% Statements
\newtheorem{thm}{Theorem}[section]
\newtheorem{lem}[thm]{Lemma}
\newtheorem{cor}[thm]{Corollary}
\newtheorem{pro}[thm]{Proposition}
\theoremstyle{definition}
\newtheorem{defi}[thm]{Definition}
\newtheorem{ex}[thm]{Example}
\newtheorem{rmk}[thm]{Remark}
\newtheorem{pdef}[thm]{Proposition-Definition}
\newtheorem{condition}[thm]{Condition}

\renewcommand{\labelenumi}{{\rm(\alph{enumi})}}
\renewcommand{\theenumi}{\alph{enumi}}

\nc{\tred}[1]{\textcolor{red}{#1}}
\nc{\tblue}[1]{\textcolor{blue}{#1}}
\nc{\tgreen}[1]{\textcolor{green}{#1}}
\nc{\tpurple}[1]{\textcolor{purple}{#1}}
\nc{\btred}[1]{\textcolor{red}{\bf #1}}
\nc{\btblue}[1]{\textcolor{blue}{\bf #1}}
\nc{\btgreen}[1]{\textcolor{green}{\bf #1}}
\nc{\btpurple}[1]{\textcolor{purple}{\bf #1}}

\nc{\ld}[1]{\textcolor{blue}{Landry:#1}}
\nc{\cm}[1]{\textcolor{red}{Chengming:#1}}
\nc{\nh}[1]{\textcolor{purple}{Norbert:#1}}
\nc{\lir}[1]{\textcolor{blue}{Li:#1}}

%%%%%%%%%%%%%% Matrix symbols.

\nc{\twovec}[2]{\left(\begin{array}{c} #1 \\ #2\end{array} \right )}
\nc{\threevec}[3]{\left(\begin{array}{c} #1 \\ #2 \\ #3 \end{array}\right )}
\nc{\twomatrix}[4]{\left(\begin{array}{cc} #1 & #2\\ #3 & #4 \end{array} \right)}
\nc{\threematrix}[9]{{\left(\begin{matrix} #1 & #2 & #3\\ #4 & #5 & #6 \\ #7 & #8 & #9 \end{matrix} \right)}}
\nc{\twodet}[4]{\left|\begin{array}{cc} #1 & #2\\ #3 & #4 \end{array} \right|}

\nc{\rk}{\mathrm{r}}
\newcommand{\g}{\mathfrak g}
\newcommand{\h}{\mathfrak h}
\newcommand{\pf}{\noindent{$Proof$.}\ }
\newcommand{\frkg}{\mathfrak g}
\newcommand{\frkh}{\mathfrak h}
\newcommand{\Id}{\rm{Id}}
\newcommand{\gl}{\mathfrak {gl}}
\newcommand{\ad}{\mathrm{ad}}
\newcommand{\add}{\frka\frkd}
\newcommand{\frka}{\mathfrak a}
\newcommand{\frkb}{\mathfrak b}
\newcommand{\frkc}{\mathfrak c}
\newcommand{\frkd}{\mathfrak d}
\newcommand {\comment}[1]{{\marginpar{*}\scriptsize\textbf{Comments:} #1}}
%%%%%%%%%%%%%%%%%%%%%%% symbols

\nc{\tforall}{\text{ for all }}

\nc{\svec}[2]{{\tiny\left(\begin{matrix}#1\\
#2\end{matrix}\right)\,}}  % column vector
\nc{\ssvec}[2]{{\tiny\left(\begin{matrix}#1\\
#2\end{matrix}\right)\,}} % subscript column vector

\nc{\typeI}{local cocycle $3$-Lie bialgebra\xspace}
\nc{\typeIs}{local cocycle $3$-Lie bialgebras\xspace}
\nc{\typeII}{double construction $3$-Lie bialgebra\xspace}
\nc{\typeIIs}{double construction $3$-Lie bialgebras\xspace}

\nc{\bia}{{$\mathcal{P}$-bimodule ${\bf k}$-algebra}\xspace}
\nc{\bias}{{$\mathcal{P}$-bimodule ${\bf k}$-algebras}\xspace}

\nc{\rmi}{{\mathrm{I}}}
\nc{\rmii}{{\mathrm{II}}}
\nc{\rmiii}{{\mathrm{III}}}
\nc{\pr}{{\mathrm{pr}}}
\newcommand{\huaA}{\mathcal{A}}

\nc{\OT}{constant $\theta$-}
\nc{\T}{$\theta$-}
\nc{\IT}{inverse $\theta$-}

\nc{\pll}{\beta}
\nc{\plc}{\epsilon}

\nc{\ass}{{\mathit{Ass}}}
\nc{\lie}{{\mathit{Lie}}}
\nc{\comm}{{\mathit{Comm}}}
\nc{\dend}{{\mathit{Dend}}}
\nc{\zinb}{{\mathit{Zinb}}}
\nc{\tdend}{{\mathit{TDend}}}
\nc{\prelie}{{\mathit{preLie}}}
\nc{\postlie}{{\mathit{PostLie}}}
\nc{\quado}{{\mathit{Quad}}}
\nc{\octo}{{\mathit{Octo}}}
\nc{\ldend}{{\mathit{ldend}}}
\nc{\lquad}{{\mathit{LQuad}}}

 \nc{\adec}{\check{;}} \nc{\aop}{\alpha}
\nc{\dftimes}{\widetilde{\otimes}} \nc{\dfl}{\succ} \nc{\dfr}{\prec}
\nc{\dfc}{\circ} \nc{\dfb}{\bullet} \nc{\dft}{\star}
\nc{\dfcf}{{\mathbf k}} \nc{\apr}{\ast} \nc{\spr}{\cdot}
\nc{\twopr}{\circ} \nc{\tspr}{\star} \nc{\sempr}{\ast}
\nc{\disp}[1]{\displaystyle{#1}}
\nc{\bin}[2]{ (_{\stackrel{\scs{#1}}{\scs{#2}}})}  %binomial coeff
\nc{\binc}[2]{ \left (\!\! \begin{array}{c} \scs{#1}\\
    \scs{#2} \end{array}\!\! \right )}  %binomial coeff
\nc{\bincc}[2]{  \left ( {\scs{#1} \atop
    \vspace{-.5cm}\scs{#2}} \right )}  %binomial coeff
\nc{\sarray}[2]{\begin{array}{c}#1 \vspace{.1cm}\\ \hline
    \vspace{-.35cm} \\ #2 \end{array}}
\nc{\bs}{\bar{S}} \nc{\dcup}{\stackrel{\bullet}{\cup}}
\nc{\dbigcup}{\stackrel{\bullet}{\bigcup}} \nc{\etree}{\big |}
\nc{\la}{\longrightarrow} \nc{\fe}{\'{e}} \nc{\rar}{\rightarrow}
\nc{\dar}{\downarrow} \nc{\dap}[1]{\downarrow
\rlap{$\scriptstyle{#1}$}} \nc{\uap}[1]{\uparrow
\rlap{$\scriptstyle{#1}$}} \nc{\defeq}{\stackrel{\rm def}{=}}
\nc{\dis}[1]{\displaystyle{#1}} \nc{\dotcup}{\,
\displaystyle{\bigcup^\bullet}\ } \nc{\sdotcup}{\tiny{
\displaystyle{\bigcup^\bullet}\ }} \nc{\hcm}{\ \hat{,}\ }
\nc{\hcirc}{\hat{\circ}} \nc{\hts}{\hat{\shpr}}
\nc{\lts}{\stackrel{\leftarrow}{\shpr}}
\nc{\rts}{\stackrel{\rightarrow}{\shpr}} \nc{\lleft}{[}
\nc{\lright}{]} \nc{\uni}[1]{\tilde{#1}} \nc{\wor}[1]{\check{#1}}
\nc{\free}[1]{\bar{#1}} \nc{\den}[1]{\check{#1}} \nc{\lrpa}{\wr}
\nc{\curlyl}{\left \{ \begin{array}{c} {} \\ {} \end{array}
    \right .  \!\!\!\!\!\!\!}
\nc{\curlyr}{ \!\!\!\!\!\!\!
    \left . \begin{array}{c} {} \\ {} \end{array}
    \right \} }
\nc{\leaf}{\ell}       % number of leafs
\nc{\longmid}{\left | \begin{array}{c} {} \\ {} \end{array}
    \right . \!\!\!\!\!\!\!}
\nc{\ot}{\otimes} \nc{\sot}{{\scriptstyle{\ot}}}
\nc{\otm}{\overline{\ot}}
\nc{\ora}[1]{\stackrel{#1}{\rar}}
\nc{\ola}[1]{\stackrel{#1}{\la}}%${\Bbb Z}$
\nc{\pltree}{\calt^\pl}
\nc{\epltree}{\calt^{\pl,\NC}}
\nc{\rbpltree}{\calt^r}
\nc{\scs}[1]{\scriptstyle{#1}} \nc{\mrm}[1]{{\rm #1}}
\nc{\dirlim}{\displaystyle{\lim_{\longrightarrow}}\,}
\nc{\invlim}{\displaystyle{\lim_{\longleftarrow}}\,}
\nc{\mvp}{\vspace{0.5cm}} \nc{\svp}{\vspace{2cm}}
\nc{\vp}{\vspace{8cm}} \nc{\proofbegin}{\noindent{\bf Proof: }}
%\nc{\proofbegin}{\begin{proof}} % AMS command
\nc{\proofend}{$\blacksquare$ \vspace{0.5cm}}
%\nc{\proofend}{\end{proof}} %AMS command
\nc{\freerbpl}{{F^{\mathrm RBPL}}}
\nc{\sha}{{\mbox{\cyr X}}}  %used to be \cyr
\nc{\ncsha}{{\mbox{\cyr X}^{\mathrm NC}}} \nc{\ncshao}{{\mbox{\cyr
X}^{\mathrm NC,\,0}}}
\nc{\shpr}{\diamond}    %Shuffle product
\nc{\shprm}{\overline{\diamond}}    %Shuffle product
\nc{\shpro}{\diamond^0}    %Shuffle product
\nc{\shprr}{\diamond^r}     %product on controlled trees
\nc{\shpra}{\overline{\diamond}^r}
\nc{\shpru}{\check{\diamond}} \nc{\catpr}{\diamond_l}
\nc{\rcatpr}{\diamond_r} \nc{\lapr}{\diamond_a}
\nc{\sqcupm}{\ot}
\nc{\lepr}{\diamond_e} \nc{\vep}{\varepsilon} \nc{\labs}{\mid\!}
\nc{\rabs}{\!\mid} \nc{\hsha}{\widehat{\sha}}
\nc{\lsha}{\stackrel{\leftarrow}{\sha}}
\nc{\rsha}{\stackrel{\rightarrow}{\sha}} \nc{\lc}{\lfloor}
\nc{\rc}{\rfloor}
\nc{\tpr}{\sqcup}
\nc{\nctpr}{\vee}
\nc{\plpr}{\star}
\nc{\rbplpr}{\bar{\plpr}}
\nc{\sqmon}[1]{\langle #1\rangle}
\nc{\forest}{\calf}
\nc{\altx}{\Lambda_X} \nc{\vecT}{\vec{T}} \nc{\onetree}{\bullet}
\nc{\Ao}{\check{A}}
\nc{\seta}{\underline{\Ao}}
\nc{\deltaa}{\overline{\delta}}
\nc{\trho}{\tilde{\rho}}

\nc{\rpr}{\circ}
%\nc{\apr}{\cdot}
\nc{\dpr}{{\tiny\diamond}}
\nc{\rprpm}{{\rpr}}

%%%%%%%%%%%%%%%%%%%%% roman fonts, in alphabetic order
\nc{\mmbox}[1]{\mbox{\ #1\ }} \nc{\ann}{\mrm{ann}}
\nc{\Aut}{\mrm{Aut}} \nc{\can}{\mrm{can}}
\nc{\twoalg}{{two-sided algebra}\xspace}
\nc{\colim}{\mrm{colim}}
\nc{\Cont}{\mrm{Cont}} \nc{\rchar}{\mrm{char}}
\nc{\cok}{\mrm{coker}} \nc{\dtf}{{R-{\rm tf}}} \nc{\dtor}{{R-{\rm
tor}}}
\renewcommand{\det}{\mrm{det}}
\nc{\depth}{{\mrm d}}
\nc{\Div}{{\mrm Div}} \nc{\End}{\mrm{End}} \nc{\Ext}{\mrm{Ext}}
\nc{\Fil}{\mrm{Fil}} \nc{\Frob}{\mrm{Frob}} \nc{\Gal}{\mrm{Gal}}
\nc{\GL}{\mrm{GL}} \nc{\Hom}{\mrm{Hom}} \nc{\hsr}{\mrm{H}}
\nc{\hpol}{\mrm{HP}} \nc{\id}{\mrm{id}} \nc{\im}{\mrm{im}}
\nc{\incl}{\mrm{incl}} \nc{\length}{\mrm{length}}
\nc{\LR}{\mrm{LR}} \nc{\mchar}{\rm char} \nc{\NC}{\mrm{NC}}
\nc{\mpart}{\mrm{part}} \nc{\pl}{\mrm{PL}}
\nc{\ql}{{\QQ_\ell}} \nc{\qp}{{\QQ_p}}
\nc{\rank}{\mrm{rank}} \nc{\rba}{\rm{RBA }} \nc{\rbas}{\rm{RBAs }}
\nc{\rbpl}{\mrm{RBPL}}
\nc{\rbw}{\rm{RBW }} \nc{\rbws}{\rm{RBWs }} \nc{\rcot}{\mrm{cot}}
\nc{\rest}{\rm{controlled}\xspace}
\nc{\rdef}{\mrm{def}} \nc{\rdiv}{{\rm div}} \nc{\rtf}{{\rm tf}}
\nc{\rtor}{{\rm tor}} \nc{\res}{\mrm{res}} \nc{\SL}{\mrm{SL}}
\nc{\Spec}{\mrm{Spec}} \nc{\tor}{\mrm{tor}} \nc{\Tr}{\mrm{Tr}}
\nc{\mtr}{\mrm{sk}}

%%%%%%%%%%%%%%%%%% bold face
\nc{\ab}{\mathbf{Ab}} \nc{\Alg}{\mathbf{Alg}}
\nc{\Algo}{\mathbf{Alg}^0} \nc{\Bax}{\mathbf{Bax}}
\nc{\Baxo}{\mathbf{Bax}^0} \nc{\RB}{\mathbf{RB}}
\nc{\RBo}{\mathbf{RB}^0} \nc{\BRB}{\mathbf{RB}}
\nc{\Dend}{\mathbf{DD}} \nc{\bfk}{{\bf k}} \nc{\bfone}{{\bf 1}}
\nc{\base}[1]{{a_{#1}}} \nc{\detail}{\marginpar{\bf More detail}
    \noindent{\bf Need more detail!}
    \svp}
\nc{\Diff}{\mathbf{Diff}} \nc{\gap}{\marginpar{\bf
Incomplete}\noindent{\bf Incomplete!!}
    \svp}
\nc{\FMod}{\mathbf{FMod}} \nc{\mset}{\mathbf{MSet}}
\nc{\rb}{\mathrm{RB}} \nc{\Int}{\mathbf{Int}}
\nc{\Mon}{\mathbf{Mon}}
%\nc{\remark}{\noindent{\bf Remark: }}
\nc{\remarks}{\noindent{\bf Remarks: }}
\nc{\OS}{\mathbf{OS}} %free operated semigroup
\nc{\Rep}{\mathbf{Rep}}
\nc{\Rings}{\mathbf{Rings}} \nc{\Sets}{\mathbf{Sets}}
\nc{\DT}{\mathbf{DT}}

%%%%%%%%%%%%%%%%%%%Bbb fonts
\nc{\BA}{{\mathbb A}} \nc{\CC}{{\mathbb C}} \nc{\DD}{{\mathbb D}}
\nc{\EE}{{\mathbb E}} \nc{\FF}{{\mathbb F}} \nc{\GG}{{\mathbb G}}
\nc{\HH}{{\mathbb H}} \nc{\LL}{{\mathbb L}} \nc{\NN}{{\mathbb N}}
\nc{\QQ}{{\mathbb Q}} \nc{\RR}{{\mathbb R}} \nc{\BS}{{\mathbb{S}}} \nc{\TT}{{\mathbb T}}
\nc{\VV}{{\mathbb V}} \nc{\ZZ}{{\mathbb Z}}

%%%%%%%%%%%%%%%%%%% cal fonts

\nc{\calao}{{\mathcal A}} \nc{\cala}{{\mathcal A}}
\nc{\calc}{{\mathcal C}} \nc{\cald}{{\mathcal D}}
\nc{\cale}{{\mathcal E}} \nc{\calf}{{\mathcal F}}
\nc{\calfr}{{{\mathcal F}^{\,r}}} \nc{\calfo}{{\mathcal F}^0}
\nc{\calfro}{{\mathcal F}^{\,r,0}} \nc{\oF}{\overline{F}}
\nc{\calg}{{\mathcal G}} \nc{\calh}{{\mathcal H}}
\nc{\cali}{{\mathcal I}} \nc{\calj}{{\mathcal J}}
\nc{\call}{{\mathcal L}} \nc{\calm}{{\mathcal M}}
\nc{\caln}{{\mathcal N}} \nc{\calo}{{\mathcal O}}
\nc{\calp}{{\mathcal P}} \nc{\calq}{{\mathcal Q}} \nc{\calr}{{\mathcal R}}
\nc{\calt}{{\mathcal T}} \nc{\caltr}{{\mathcal T}^{\,r}}
\nc{\calu}{{\mathcal U}} \nc{\calv}{{\mathcal V}}
\nc{\calw}{{\mathcal W}} \nc{\calx}{{\mathcal X}}
\nc{\CA}{\mathcal{A}}

%%%%%%%%%%%%%%%%%%  frak fonts
\nc{\fraka}{{\mathfrak a}} \nc{\frakB}{{\mathfrak B}}
\nc{\frakb}{{\mathfrak b}} \nc{\frakd}{{\mathfrak d}}
\nc{\oD}{\overline{D}}
\nc{\frakF}{{\mathfrak F}} \nc{\frakg}{{\mathfrak g}}
\nc{\frakm}{{\mathfrak m}} \nc{\frakM}{{\mathfrak M}}
\nc{\frakMo}{{\mathfrak M}^0} \nc{\frakp}{{\mathfrak p}}
\nc{\frakS}{{\mathfrak S}} \nc{\frakSo}{{\mathfrak S}^0}
\nc{\fraks}{{\mathfrak s}} \nc{\os}{\overline{\fraks}}
\nc{\frakT}{{\mathfrak T}}
\nc{\oT}{\overline{T}}
%\nc{\frakx}{{\mathfrak x}}
\nc{\frakX}{{\mathfrak X}} \nc{\frakXo}{{\mathfrak X}^0}
\nc{\frakx}{{\mathbf x}}
%\nc{\frakTxo}{{\frakTx}^0}
\nc{\frakTx}{\frakT}      %All rooted trees, correspond to \ncsha(X)
\nc{\frakTa}{\frakT^a}        % rooted trees for \ncsha(A)
\nc{\frakTxo}{\frakTx^0}   % rooted trees for \ncshao(X)
\nc{\caltao}{\calt^{a,0}}   % rooted trees for \ncshao(A)
\nc{\ox}{\overline{\frakx}} \nc{\fraky}{{\mathfrak y}}
\nc{\frakz}{{\mathfrak z}} \nc{\oX}{\overline{X}}

\font\cyr=wncyr10

\nc{\al}{\alpha}
\nc{\lam}{\lambda}
\nc{\lr}{\longrightarrow}
\newcommand{\K}{\mathbb {K}}
\newcommand{\A}{\rm A}

%%%%%%%%%%%%%%%%%%%%%%%%%%%%%%%%%%%%%%%%%%%%%%%%%%%%%%%%%%%%%%%%%%

\title[Anti-flexible bialgebras]{Anti-flexible bialgebras}

\author[Mafoya Landry Dassoundo]{Mafoya Landry Dassoundo$^\star$}
\address[$^\star$]{Chern Institute of Mathematics \& LPMC, Nankai University, Tianjin 300071, China }
\email{dassoundo@yahoo.com}

\author[Chengming Bai]{Chengming Bai$^\dag$}
\address[$^\dag$]{Chern Institute of Mathematics \& LPMC, Nankai University, Tianjin 300071, China }
\email{baicm.nankai.edu.cn}

\author[Mahouton Norbert Hounkonnou]{Mahouton Norbert Hounkonnou$^\ddag$}
\address[$^\ddag$]{University of Abomey-Calavi, International Chair in Mathematical Physics and Applications, ICMPA-UNESCO Chair, 072 BP 50, Cotonou, Rep. of Benin}
\email{hounkonnou@yahoo.fr}

%\maketitle
%\large

\begin{abstract}
We establish a bialgebra theory for anti-flexible algebras in this
paper. We introduce the notion of an anti-flexible bialgebra which
is equivalent to a Manin triple of anti-flexible algebras. The
study of a special case of anti-flexible bialgebras leads to the
introduction of anti-flexible Yang-Baxter equation in an
anti-flexible algebra which is an analogue of the classical
Yang-Baxter equation in a Lie algebra or the associative
Yang-Baxter equation in an associative algebra. It is a unexpected
consequence that both the anti-flexible Yang-Baxter equation and
the associative Yang-Baxter equation have the same form. A
skew-symmetric solution of anti-flexible Yang-Baxter equation
gives an anti-flexible bialgebra. Finally the notions of an
$\mathcal O$-operator of an anti-flexible algebra and a pre-anti-flexible algebra are
introduced to construct skew-symmetric solutions of anti-flexible
Yang-Baxter equation.
\end{abstract}

\subjclass[2010]{17A20,  17D25, 16D20, 16T10, 16T15,  16T25}

\keywords{anti-flexible algebra, anti-flexible bialgebra,
anti-flexible  Yang-Baxter equation, $\mathcal{O}$-operator}

\maketitle

\tableofcontents

\numberwithin{equation}{section}

\tableofcontents
\numberwithin{equation}{section}
\allowdisplaybreaks

% % % % % % % % % % % % % % % % % % %
% % % % % % % % % % % % % % % % % % %
\section{Introduction}

At first, we recall the definition of a flexible algebra.
\begin{defi} Let $A$ be a vector space over a field $\mathbb F$ equipped with a bilinear product $(x,y)\rightarrow xy$. Set the associator as
\begin{equation}
(x,y,z)=(xy)z-x(yz),\forall x,y,z\in A.\label{eq:asso}
\end{equation}
$A$ is called a {\bf flexible algebra} if the following identity is satisfied
\begin{equation}
(x,y,x)=0,\;\;{\rm or}\;\;{\rm equivalently},\;\;(xy)x=x(yx),\;\;\forall x,y\in A.\label{eq:fa}
\end{equation}
\end{defi}

As a natural generalization of associative algebras, flexible
algebras were studied widely. For example, using the solvability
and reducibility of the radicals of their underlying Lie algebras,
finite-dimensional flexible Lie-admissible algebras were
characterized  in \cite{Benkart_O};
any simple strictly
power-associative algebra of characteristic prime to $6$ of degree
greater than $2$ is a flexible algebra (\cite{Kosier}).
Note that the
``linearization" of the identity ~\eqref{eq:fa} gives the
following equivalent identity by substituting $x+z$ for $x$ in
Eq.~\eqref{eq:fa}:
\begin{equation}\label{eq:fa1}
(x,y,z)+(z,y,x)=0,\;\;\forall x,y,z\in A.
\end{equation}

It is also natural to consider certain generalization of flexible
algebras which leads to the introduction of several classes of
nonassociative algebras \cite{Rodabaugh_3}. In particular, the
so-called anti-flexible algebras were introduced as follows.
\begin{defi} Let $A$ be a vector space equipped with a bilinear product $(x,y)\rightarrow xy$.
$A$ is called an {\bf anti-flexible algebra} if the following identity is satisfied
\begin{equation}
(x,y,z)=(z,y,x),\;\;{\rm or}\;\;{\rm equivalently},\;\;(xy)z-x(yz)=(zy)x-z(yx),\;\;\forall x,y,z\in A.\label{eq:afa}
\end{equation}
\end{defi}
Note that the identity~\eqref{eq:afa} means that the associator~\eqref{eq:asso} is symmetric in $x,z$ and thus an anti-flexible algebra is also
called a {\bf center-symmetric algebra} in \cite{Hounkonnou_D_CSA} (it is also called a {\bf $G_4$-associative algebra} in \cite{ME}).
 The study of anti-flexible algebras is fruitful, too.
For example, simplicity and semi-simplicity of anti-flexible
algebras were investigated in \cite{Rodabaugh_1}; the simple,
semisimple (totally) anti-flexible algebras over splitting fields
of characteristic different to $2$ and $3$ were studied and
classified in \cite{Bhandari,Rodabaugh_2,Rodabaugh}; the
primitive structures and prime anti-flexible rings were
investigated in \cite{Celik}; furthermore, it were shown that a
simple nearly anti-flexible algebra of characteristic prime to $30$
satisfying the identity $(x, x, x) = 0$ in which its commutator
gives non-nilpotent structure possesses a unity element
(\cite{Davis_R}).

On the other hand, a bialgebra structure on a given algebraic
structure is obtained as a coalgebra structure together which gives the same algebraic structure on the dual space with a set
of compatibility conditions between the multiplications and
comultiplications. One of the most famous examples of bialgebras is
the Lie bialgebra (\cite{Drinfeld}) and more importantly there have been a
lot of bialgebra theories for other algebra structures that essentially follow the approach of
Lie bialgebras  such as antisymmetric infinitesimal
bialgebras (\cite{Aguiar_1, Bai_Double}), left-symmetric
bialgebras (\cite{ Bai_LSA}), alternative D-bialgebras
(\cite{Gon}) and Jordan bialgebras (\cite{Zhelyabin}).

In this paper, we give a bialgebra theory for anti-flexible
algebras. We still take a similar approach as of the study on Lie
bialgebras, that is, the compatibility condition is still decided
by an analogue of Manin triple of Lie algebras, which we call a
Manin triple of anti-flexible algebras. The notion of
anti-flexible bialgebra is thus introduced as an equivalent
structure of a Manin triple of anti-flexible algebras, which is
interpreted in terms of matched pairs of anti-flexible algebras.
Here the dual bimodule of a bimodule of an anti-flexible algebra
plays an important role. We would like to point out that both
anti-flexible and associative algebras have the same forms of dual
bimodules, which is quite different from other generalizations of
associative algebras such as left-symmetric algebras (\cite{
Bai_LSA}) or other $G$-associative algebras in \cite{ME}.

Although to our knowledge, a well-constructed cohomology theory
for anti-flexible algebras is unknown yet, we still consider a
special case of anti-flexible bialgebras 
following the study of coboundary Lie
bialgebras for Lie algebras (\cite{Drinfeld}) or coboundary
antisymmetric infinitesimal bialgebras for associative algebras
(\cite{Bai_Double}). The study of 
such a class of anti-flexible
bialgebras also leads to the introduction of anti-flexible
Yang-Baxter equation in an anti-flexible equation which is an
analogue of the classical Yang-Baxter equation in a Lie algebra or
the associative Yang-Baxter equation in an associative algebra. A
skew-symmetric solution of anti-flexible Yang-Baxter equation
gives an 
anti-flexible bialgebra.

There is an unexpected consequence that both the anti-flexible Yang-Baxter
equation and the associative Yang-Baxter equation have the same form.
It is partly due to the fact that both anti-flexible and
associative algebras have the same forms of dual bimodules.
Therefore some properties of anti-flexible Yang-Baxter
equation can be obtained directly from the corresponding ones of
associative Yang-Baxter equation.

In particular, as for the study on the associative Yang-Baxter equation, in
order to obtain skew-symmetric solutions of anti-flexible
Yang-Baxter equation, we introduce the notions of an $\mathcal
O$-operator of an anti-flexible algebra which is an analogue of an
$\mathcal O$-operator of a Lie algebra introduced by Kupershmidt in \cite{Kupershmidt__} as a
 natural generalization of the classical Yang-Baxter equation in a
 Lie algebra, and a pre-anti-flexible algebra.
The former gives a construction of skew-symmetric solutions of
anti-flexible Yang-Baxter equation  in a semi-direct product
anti-flexible algebra, whereas the latter as a generalization of a
dendriform algebra (\cite{Loday}) gives a bimodule of the
associated anti-flexible algebra such that the identity is a
natural $\mathcal O$-operator associated to it. Therefore a
construction of skew-symmetric solutions of anti-flexible
Yang-Baxter equation and hence anti-flexible bialgebras from
pre-anti-flexible algebras is given. Note that from the point of view of operads,
pre-anti-flexible algebras are the splitting of anti-flexible algebras (\cite{BBGN,PBG}).

The paper is organized as follows. In Section 2, we study
bimodules and matched pairs of anti-flexible algebras. In
particular, we give the dual bimodule of a bimodule of an
anti-flexible algebra. In Section 3, we give the
notion of a Manin triple of anti-flexible algebras and then
interpret it in terms of matched pairs of anti-flexible
algebras. The notion of an anti-flexible bialgebra is thus introduced as
an equivalent structure of a Manin triple of anti-flexible
algebras. In Section 4, we consider the 
special class of anti-flexible
bialgebras which lead to the introduction of anti-flexible
Yang-Baxter equation. A skew-symmetric solution of
anti-flexible Yang-Baxter equation gives such a 
anti-flexible bialgebra. In Section 5, we introduce the notions of
an $\mathcal O$-operator of an anti-flexible algebra and a pre-anti-flexible algebra. The
relationships between them and the anti-flexible Yang-Baxter
equation are given. In particular, we give constructions of
skew-symmetric solutions of anti-flexible Yang-Baxter equation
from $\mathcal O$-operators of anti-flexible algebras and pre-anti-flexible algebras.

Throughout this paper, all vector spaces are
finite-dimensional over a base field $\mathbb F$ whose
characteristic is not $2$, although many results still hold in the
infinite dimension.

\section{Bimodules and matched pairs of anti-flexible algebras}

In this section, we first introduce the notion of a bimodule of an anti-flexible algebra. Then we study the dual bimodule of a bimodule
of an anti-flexible algebra. We also give the notion of a matched pair of anti-flexible algebras.

%%%%%%%%%%%%%%%%%%%%%%%%%%%%%%%%%%%%%%%%%%%%%%%%%%%%%%%%%%%%%%%%%%%%%%%%
%%%%%%%%%%%%%%%%%%%%%%%%%%%%%%%%%%%%%%%%%%%%%%%%%%%%%%%%%%%%%%%%%%%%%%%%
%%%%%%%%%%%%%%%%%%%%%%%%%%%%%%%%%%%%%%%%%%%%%%%%%%%%%%%%%%%%%%%%%%%%%%%%
%%%%%%%%%%%%%%%%%%%%%%%%%%%%%%%%%%%%%%%%%%%%%%%%%%%%%%%%%%%%%%%%%%%%%%%%
%%%%%%%%%%%%%%%%%%%%%%%%%%%%%%%%%%%%%%%%%%%%%%%%%%%%%%%%%%%%%%%%%%%%%%%%
\begin{defi}\label{bimodule}
Let $(A, \cdot)$ be an anti-flexible algebra and $V$ be a vector space.
Let $\displaystyle l,r : A \rightarrow {\rm End}(V)$ be two linear maps.
If for any  $x, y \in A$,
\begin{eqnarray}\label{eqbimodule1}
l{(x\cdot y)}-l(x)l(y)=r(x)r(y)-r({y\cdot x}),
\end{eqnarray}
\begin{eqnarray}\label{eqbimodule2}
    l(x)r(y)-r(y)l(x)=l(y)r(x)-r(x)l(y),
\end{eqnarray}
then it is called a {\bf bimodule} of  $(A, \cdot)$, denoted by
$\displaystyle(l, r, V)$.
Two bimodules $(l_1,r_1,V_1)$ and $(l_2,r_2,V_2)$ of an anti-flexible algebra $A$ is called {\bf equivalent} if there exists
a linear isomorphism $\varphi:V_1\rightarrow V_2$ satisfying
\begin{equation}
\varphi l_1(x) =l_2(x)\varphi, \varphi r_1(x)=r_2(x)\varphi,\;\;\forall x\in A.
\end{equation}
\end{defi}

\begin{rmk}
Note that if both sides of Eqs.~(\ref{eqbimodule1}) and (\ref{eqbimodule2}) are zero,
then they exactly give the definition of a bimodule of an associative algebra.
\end{rmk}

Let $(A,\cdot)$ be an anti-flexible algebra. For any $x,y\in A$, let $L_x$ and
$R_x$ denote the left and right multiplication operators
respectively, that is, $L_x(y)=xy$ and $R_x(y)=yx$. Let $L,R: A\rightarrow {\rm End}(A)$ be two
linear maps with $x\rightarrow L_x$ and $x\rightarrow R_x$ for any $x\in A$ respectively.

\begin{ex}
Let $(A,\cdot)$ be an anti-flexible algebra. Then $(L,R,A)$ is a bimodule of $(A,\cdot)$, which is called
the {\bf regular bimodule of $(A,\cdot)$}.
\end{ex}

\begin{pro}\label{Propo_bimodule}
Let $(A, \cdot)$ be an anti-flexible algebra and $V$ be a vector space.
Let $\displaystyle l,r : A \rightarrow {\rm End}(V)$ be two linear maps.
Then $(l, r, V)$ is a bimodule of $(A, \cdot)$ if and only if the direct sum
$A\oplus V $ of vector spaces is turned into an anti-flexible algebra  by defining the multiplication in $A\oplus V$ by
\begin{eqnarray}
(x+u)\ast (y+v) = x\cdot y+l({x})v+r({y})u,\;\;\forall x,y\in A, u,v\in V.
\end{eqnarray}
We call it {\bf semi-direct product} and denote it by
$\displaystyle A \ltimes_{l, r} V $ or simply $A \ltimes V.$
\end{pro}

\begin{proof}
It is straightforward or follows from Theorem~\ref{theoo} as a direct consequence.
\end{proof}

It is known that an anti-flexible algebra is a Lie-admissible algebra (\cite{ME}).

\begin{pro}
Let $(A, \cdot)$ be an anti-flexible algebra. Define the commutator by
\begin{equation}
[x, y]=x\cdot y-y\cdot x,\;\;\forall x,y\in A.
\end{equation}
Then it is a Lie algebra and we denote it by $(\frak g(A), [\;,\;])$ or simply $\frak g(A)$, which is called {\bf the associated Lie algebra of $(A,\cdot)$}.
\end{pro}

\begin{cor}\label{propcentlie}
Let $(l,r,V)$ be a bimodule of an anti-flexible algebra $(A,\cdot)$. Then $(l-r,V)$ is a representation
of the associated Lie algebra $(\frak g(A),[\;,\;])$.
\end{cor}

\begin{proof}
For any $x,y\in A$, we have
\begin{eqnarray*}
[(l-r)(x),(l-r)(y)]&=&[l(x),l(y)]+[r(x),r(y)]-[l(x),r(y)]-[r(x),l(y)]\\
&=&[l(x),l(y)]+[r(x),r(y)]
=l(x\cdot y-y\cdot x)-r(x\cdot y-y\cdot x)\\
&=&(l-r)([x,y]).
\end{eqnarray*}
Hence $(l-r,V)$ is a representation
of $(\frak g(A),[\;,\;])$.
\end{proof}

Let $(A,\cdot)$ be an anti-flexible algebra. Let $V$ be a vector space and $\alpha:A\rightarrow {\rm End}(V)$ be a linear map.
Define a linear map $\alpha^*:A\rightarrow {\rm End}(V^*)$ as
\begin{equation}
\langle \alpha^*(x) u^*,v\rangle=\langle
u^*,\alpha(x)v\rangle,\;\;\forall x\in A, v\in V, u^*\in V^*,
\end{equation}
where $\langle,\rangle$ is the usual pairing between $V$ and the
dual space $V^*$.
\begin{pro}
Let $(l,r,V)$ be a bimodule of an anti-flexible algebra $(A,\cdot)$. Then $(r^*,l^*,V^*)$ is bimodule of $(A,\cdot)$.
\end{pro}

\begin{proof}
For all $x,y\in A, u^*\in V^*, v\in V$, we have
\begin{eqnarray*}
\left<(r^*{(x\cdot y)}-r^*(x)r^*(y))u^*, v\right>
        &=&
        \left<u^*, (r{(x\cdot y)}-r(y)r(x))(v)\right>
        = \left<u^*,(l(y)l(x)-l(y\cdot x))(v)\right> \\
&=&\left<(l^*(x)l^*(y)-l^*(y\cdot x))u^*, v\right>;\\
\langle (l^*(x)r^*(y)-r^*(y)l^*(x))u^*,v\rangle
&=&\langle u^*, (r(y)l(x)-l(x)r(y))(v)\rangle
=\langle u^*, (r(x)l(y)-l(y)r(x))(v)\rangle\\
&=&\langle (l^*(y)r^*(x)-r^*(x)l^*(y))u^*, v\rangle.
\end{eqnarray*}
Hence $(r^*,l^*,V^*)$ is bimodule of $(A,\cdot)$.
\end{proof}

\begin{rmk}\label{rmk:same}
Note that for a bimodule $(l,r,V)$ of an associative algebra, $(r^*,l^*,V^*)$ is also
a bimodule.  Therefore, for both associative and anti-flexible algebras, the ``dual bimodules" in the above
sense have the same form.
\end{rmk}

\begin{thm}\label{theoo}{\rm (\cite{Hounkonnou_D_CSA})}
    Let $(A, \cdot)$ and $(B, \circ)$ be two anti-flexible algebras.
    Suppose that there are four linear maps $l_A,r_A:A\rightarrow {\rm End}(B)$ and
$l_B,r_B:B\rightarrow {\rm End}(A)$ such that $(l_{A}, r_{A}, B)$ and $(l_{B}, r_{B}, A)$
    are bimodules of $(A,\cdot)$ and $(B,\circ)$ respectively, obeying the following relations:
\begin{eqnarray}\label{eqq1}
     l_{B}(a)(x\cdot y) +r_{B}(a)(y\cdot x)-r_{B}(l_{A}(x)a)y-  y\cdot(r_{B}(a)x) -l_{B}(r_{A}(x)a)y -  (l_{B}(a)x)\cdot y  = 0,
\end{eqnarray}
\begin{eqnarray}\label{eqq2}
 l_{A}(x)(a\circ b) +r_{A}(x)(b\circ a)-r_{A}(l_{B}(a)x)b- b\circ (r_{A}(x)a)+l_{A}(r_{B}(a)x)b - (l_{A}(x)a)\circ b=0,
\end{eqnarray}
\begin{eqnarray}\label{eqq3}
\begin{array}{lll}
y\cdot (l_{B}(a)x)+(r_{B}(a)x)\cdot y - (r_{B}(a)y)\cdot x-l_{B}(l_{A}(y)a)x \cr+
r_{B}(r_{A}(x)a)y+l_{B}(l_{A}(x)a)y  -x\cdot (l_{B}(a)y)-r_{B}(r_{A}(y)a)x=0,
\end{array}
\end{eqnarray}
\begin{eqnarray}\label{eqq4}
\begin{array}{lll}
b \circ (l_{A}(x)a)+(r_{A}(x)a)\circ b -(r_{A}(x)b)\circ a-l_{A}(l_{B}(b)x)a\cr+  r_{A}(r_{B}(a)x)b+l_{A}(l_{B}(a)x)b  -a\circ (l_{A}(x)b)  -r_{A}(r_{B}(b)x)a=0, \end{array}
\end{eqnarray}
for any $x,y\in A, a,b\in B$. Then there is an anti-flexible algebra structure on $A \oplus B$
    given by:
\begin{equation}\label{product}
 (x+a)\ast (y+b)= (x \cdot y + l_{B}(a)y+r_{B}(b)x)+ (a \circ b + l_{A}(x)b+r_{A}(y)a),\;\;\forall x,y\in A, a,b\in B.
\end{equation}
Conversely,  every
anti-flexible algebra which is a direct sum of the underlying
vector spaces of two subalgebras can be obtained from the above way.
\end{thm}

\begin{defi} Let $(A, \cdot)$ and $(B, \circ)$ be two anti-flexible algebras. Suppose that
there are four linear maps $l_A,r_A:A\rightarrow {\rm End}(B)$ and
$l_B,r_B:B\rightarrow {\rm End}(A)$ such that $(l_{A}, r_{A}, B)$ and $(l_{B}, r_{B}, A)$
    are bimodules of $(A,\cdot)$ and $(B,\circ)$ and Eqs.~\eqref{eqq1}-\eqref{eqq4} hold. Then
we call the six-tuple $(A, B, l_{A}, r_{A}, l_{B}, r_{B})$ a {\bf matched pair of anti-flexible algebras}.
We also denote the anti-flexible algebra defined by Eq.~(\ref{product}) by  $\displaystyle A \bowtie_{l_{B}, r_{B}}^{l_{A}, r_{A}} B$ or
simply by $\displaystyle A \bowtie B$.
\end{defi}

\section{Manin triples of anti-flexible algebras and anti-flexible bialgebras}

In this section, we introduce the notions of a Manin triple of anti-flexible algebras and
an anti-flexible bialgebra. The equivalence between them is interpreted in terms of matched pairs of anti-flexible
algebras.

\begin{defi}
 A bilinear form $\frak B$ on an anti-flexible algebra $(A,\cdot)$ is called {\bf invariant} if
\begin{equation}
\frak B(x\cdot y,z)=\frak B(x, y\cdot z),\;\;\forall x,y,z\in A.
\end{equation}
\end{defi}

\begin{pro}\label{pro:dual}
Let $(A,\cdot)$ be an anti-flexible algebra. If there is a
nondegenerate symmetric invariant bilinear form $\frak B$ on $A$,
then as bimodules of the anti-flexible algebra $(A,\cdot)$, $(L,R,
A)$ and $(R^*,L^*, A^*)$ are equivalent. Conversely, if as
bimodules of an anti-flexible algebra $(A,\cdot)$, $(L,R, A)$ and
$(R^*,L^*, A^*)$ are equivalent, then there exists a nondegenerate
invariant bilinear form $\frak B$ on $A$.
\end{pro}

\begin{proof}
Since $\frak B$ is nondegenerate, there exists a linear isomorphism $\varphi: A\rightarrow A^*$ defined by
$$\langle \varphi (x), y\rangle =\frak B(x,y),\;\;\forall x, y\in A.$$
Hence for any $x,y,z\in A$, we have
\begin{eqnarray*}
\langle \varphi L(x) y, z\rangle&=&\frak B(x\cdot y, z)=\frak B(z, x\cdot y)=\frak B(z\cdot x, y)=\langle \varphi (y), z\cdot x\rangle
=\langle R^*(x)\varphi(y), z\rangle;\\
\langle \varphi R(x) y, z\rangle&=&\frak B(y\cdot x, z)=\frak B(y, x\cdot z)=\langle \varphi (y), x\cdot z\rangle
=\langle L^*(x)\varphi(y), z\rangle.
\end{eqnarray*}
Hence $(L,R, A)$ and $(R^*,L^*, A^*)$ are equivalent. Conversely,
by a similar way, we can get the conclusion.
\end{proof}

\begin{defi} A {\bf Manin triple of
anti-flexible algebras} is a
triple of anti-flexible algebras $(A,A^{+},A^{-})$ together with a
nondegenerate symmetric invariant bilinear form $\mathfrak{B}$ on $A$ such that the following conditions are satisfied.
\begin{enumerate}
 \item $A^{+}$ and $A^{-}$
are anti-flexible subalgebras of $A$;
 \item
  $A=A^{+}\oplus A^{-}$ as
vector spaces; \item  $A^{+}$ and $A^{-}$ are isotropic with
respect to $\mathfrak{B}$, that is,
$\mathfrak{B}(x_+,y_+)=\mathfrak{B}(x_-,y_-)=0$ for any
$x_+,y_+\in A^+,x_-,y_-\in A^-$.
\end{enumerate}
A {\bf isomorphism} between two Manin triples $(A,A^{+},A^{-})$
and $(B,B^{+},B^{-})$ of anti-flexible algebras
 is an isomorphism $\varphi:A\rightarrow B$ of anti-flexible algebras  such that
\begin{equation}\varphi(A^{+})= B^{+},\;
\varphi(A^{-})=B^{-},\;
\mathfrak{B}_A(x,y)=\mathfrak{B}_B(\varphi(x), \varphi(y)),\;
\forall x,y\in A.\end{equation}
\end{defi}
\begin{defi} Let  $(A,\cdot)$ be an anti-flexible algebra. Suppose that ``$\circ$"  is an
anti-flexible algebra structure on the dual space $A^*$ of $A$ and
there is an anti-flexible algebra structure on the direct sum $A\oplus A^*$ of
the underlying vector spaces of $A$ and $A^*$ such that
$(A,\cdot)$ and $(A^*,\circ)$ are subalgebras and the natural
symmetric bilinear form on $A\oplus A^*$ given by
\begin{equation}\label{eq:sbl}\mathfrak{B}_d(x+a^*,y+b^*):=\langle  a^*,y\rangle +\langle
x,b^*\rangle ,\; \forall x,y\in A; a^*,b^*\in A^*,\end{equation} is
invariant, then $(A\oplus A^*,A,A^*)$ is called a
{\bf standard Manin triple of anti-flexible algebras associated to
$\mathfrak{B}_d$.}
\end{defi}

 Obviously, a standard Manin triple of anti-flexible algebras is a Manin triple of anti-flexible algebras. Conversely,
we have

\begin{pro}
Every Manin triple of anti-flexible algebras is isomorphic to a
standard one.
\end{pro}

\begin{proof} Since in this case $A^{-}$ and $(A^+)^*$ are identified by
the nondegenerate invariant bilinear form, the anti-flexible algebra structure on
$A^{-}$ is transferred to  $(A^+)^*$. Hence the anti-flexible algebra
structure on $A^{+}\oplus A^{-}$ is transferred to
$A^{+}\oplus(A^{+})^*$. Then the conclusion holds.\end{proof}

\begin{pro} \label{Propo} Let $(A,\cdot)$ be an anti-flexible algebra.
Suppose that there is an anti-flexible algebra structure
``$\circ$" on the dual space $A^*$. Then there exists an
anti-flexible algebra structure on the vector space $A\oplus A^*$
such that  $(A\oplus A^*, A,A^*)$ is a standard  Manin triple of
anti-flexible algebras associated to $\mathfrak{B}_d$ defined by
Eq.~\eqref{eq:sbl} if and only if $(A, A^*, R_{\cdot}^*,
L_{\cdot}^*, R_{\circ}^*, L_{\circ}^* )$ is a matched pair of
anti-flexible algebras.
\end{pro}

\begin{proof}
It follows from the same proof of \cite[Theorem
2.2.1]{Bai_Double}.
\end{proof}

\begin{pro}\label{Theo}
Let $(A, \cdot )$ be an anti-flexible algebra. Suppose that there exists an anti-flexible algebra structure $``\circ "$ on the dual space $A^{*}$.  Then $(A, A^{*}, R_{\cdot}^*, L_{\cdot }^*, R_{\circ}^*, L_{\circ}^*)$ is a matched pair of anti-flexible algebras
if and only if for any $x,y\in A, a\in A^*$,
\begin{equation}\label{eq_matched1}
-R^*_{\circ}(a)(x\cdot y)-L^*_{\circ}(a)(y\cdot x)+L_{\circ}^*(R_{\cdot}^*(x)a)y+
y\cdot (L_{\circ}^*(a)x)+R_{\circ}^*(L_{\cdot}^*(x)a)y+(R_{\circ}^*(a)x)\cdot y=0,
\end{equation}
\begin{eqnarray}\label{eq_matched2}
\begin{array}{lll}
y\cdot (R_{\circ}^*(a)x)-x\cdot (R_{\circ}^*(a)y)+(L_{\circ}^*(a)x) \cdot y-(L_{\circ}^*(a)y)\cdot x\cr
+L_{\circ}^*(L_{\cdot}^*(x)a)y-R_{\circ}^*(R_{\cdot}^*(y)a)x +
R_{\circ}^*(R_{\cdot}^*(x)a)y-L_{\circ}^*(L_{\cdot}^*(y)a)x=0.
\end{array}
\end{eqnarray}
\end{pro}

\begin{proof}
Obviously, Eq.~(\ref{eq_matched1}) is exactly Eq.~(\ref{eqq2})
and Eq.~(\ref{eq_matched2}) is exactly Eq.~(\ref{eqq4})
in the case
$l_A=R^*_\cdot, r_A=L^*_\cdot$,
$l_B=l_{A^*}=R^*_{\circ},r_B=r_{A^*}=L^*_\circ$.
 For any $x,y\in A, a, b\in A^*$, we have:
\begin{eqnarray*}
&&\left<R^*_{\circ}(a)(x \cdot y), b    \right>=\left< x\cdot y, R_{\circ}(a)b   \right>
=\left< x\cdot y , b\circ a   \right>=\left< L_{\cdot}(x)y , b\circ a  \right>=
\left<y, L_{\cdot}^*(x)(b\circ a)    \right>;\\
%%%%
%%%%
&&\left<L^*_{\circ}(a)(y\cdot x), b  \right>=\left< y\cdot x, L_{\circ}(a)b   \right>
=\left<y\cdot x, a\circ b    \right>=\left< R_{\cdot}(x) y, a\circ b \right>=
\left<y, R^*_{\cdot}(x)(a\circ b)    \right>;\\
%%%%
%%%%
&&\left< L_{\circ}^*(R_{\cdot}^*(x)a)y, b   \right>=\left< y,L_{\circ}(R_{\cdot}^*(x)a)b   \right>
=\left<y, (R_{\cdot}^*(x)a)\circ b    \right>;\\
%%%%
%%%%
&&\left<y\cdot (L_{\circ}^*(a)x), b    \right>=
\left< R_{\cdot}(L_{\circ}^*(a)x)y, b   \right>=
\left< y,  R_{\cdot}^*(L_{\circ}^*(a)x)b \right>;\\
&&\left< R_{\circ}^*(L_{\cdot}^*(x)a)y, b   \right>=
\left< y, R_{\circ}(L_{\cdot}^*(x)a)b   \right>=
\left< y, b\circ(L_{\cdot}^*(x)a)    \right>;\\
&&\left<(R_{\circ}^*(a)x)\cdot y, b    \right>=
\left<  L_{\cdot}(R_{\circ}^*(a)x)y, b  \right>=
\left<y, L_{\cdot}^*(R_{\circ}^*(a)x)b    \right>.
\end{eqnarray*}
Then Eq.~(\ref{eqq1}) holds if and only if Eq.~(\ref{eqq2}) holds.
Similarly, Eq.~(\ref{eqq3}) holds if and only if Eq.~(\ref{eqq4}) holds.
Therefore the conclusion holds.
\end{proof}

 Let $V$ be a vector space. Let $\sigma: V\otimes V \rightarrow
V\otimes V$ be the {\it flip} defined as
\begin{equation}\sigma(x \otimes y) = y\otimes x,\quad
\forall x, y\in V.
\end{equation}

\begin{thm}\label{thm:bialgebra}
Let $(A,\cdot)$ be an anti-flexible algebra. Suppose there is an
anti-flexible algebra structure $``\circ"$ on its dual space $A^*$
given by a linear map $\Delta^*: A^*\otimes A^*\rightarrow A^*$.
Then $(A,A^*, R_\cdot^*,L_\cdot^*,R_\circ^*,L_\circ^*)$ is a matched
pair of anti-flexible algebras if and only if $\Delta:A\rightarrow
A\otimes A $ satisfies the following two conditions:
\begin{equation}\label{eq_bialg1}
\Delta( x\cdot y)+\sigma\Delta( y\cdot x)=(\sigma(\id \otimes  L_{\cdot}(y))+R_{\cdot}(y)\otimes \id)\Delta (x)+
(\sigma( R_{\cdot}(x)\otimes \id)+\id\otimes L_{\cdot}(x))\Delta (y),
\end{equation}
\begin{equation}\label{eq_bialg2}
\begin{array}{lll}
(\sigma(\id \otimes R_{\cdot}(y))-\id\otimes R_{\cdot}(y)-
\sigma(L_{\cdot}(y)\otimes\id)
+L_{\cdot}(y)\otimes\id)\Delta(x)=\cr
(\sigma(\id\otimes R_{\cdot}(x))
-\id\otimes R_{\cdot}(x)
-\sigma(L_{\cdot}(x)\otimes\id)+
L_{\cdot}(x)\otimes\id
)\Delta(y),
\end{array}
\end{equation}
for any $x,y\in A$.
\end{thm}

\begin{proof}
For any $x,y\in A$ and any $a,b\in A^*$, we have
\begin{eqnarray*}
&&\langle \Delta(x\cdot y), a\otimes b\rangle = \langle x\cdot y,
a\cdot b\rangle, =\langle L_{\circ}^*(a)(x\cdot y), b\rangle, \cr
&&\langle \sigma\Delta(y\cdot x), a\otimes b\rangle = \langle
y\cdot x, b\circ a\rangle = \langle R_{\circ}^*(a)(y\cdot x),
b\rangle, \cr &&\langle \sigma(\id\otimes L_{\cdot}(y))\Delta(x),
a\otimes b\rangle = \langle x, b\circ(L_{\cdot}^*(y)a)\rangle =
\langle R_{\circ}^*(L_{\cdot}^*(y)a)x, b\rangle, \cr &&\langle
(R_{\cdot}(y)\otimes\id)\Delta(x), a\otimes b\rangle = \langle x,
(R_{\cdot}^*(y)a)\circ b\rangle = \langle
L_{\circ}^*(R_{\cdot}^*(y)a)x,b\rangle, \cr &&\langle
\sigma(R_{\cdot}(x)\otimes \id)\Delta(y), a\otimes b\rangle =
\langle y, (R_{\cdot}^*(x)b)\circ a\rangle = \langle
(R_{\circ}^*(a)y)\cdot x, b\rangle, \cr &&\langle (\id\otimes
L_{\cdot}(x))\Delta(y), a\otimes b\rangle = \langle y, a\circ
(L_{\cdot}^*(x)b)\rangle = \langle x\cdot(L_{\circ}^*(a)y),
b\rangle.
\end{eqnarray*}
Then Eq.~\eqref{eq_matched1} is equivalent to
Eq.~\eqref{eq_bialg1}. Moreover, we have
\begin{eqnarray*}
&&\langle \sigma(\id\otimes R_{\cdot}(y))\Delta(x), a\otimes
b\rangle = \langle x, b\circ (R_{\cdot}^*(y)a)\rangle = \langle
R_{\circ}^*(R_{\cdot}^*(y)a)x, b\rangle, \cr &&\langle (\id\otimes
R_{\cdot}(y))\Delta(x), a\otimes b\rangle = \langle x, a\circ
(R_{\cdot}^*(y)b)\rangle = \langle (L_{\circ}^*(a)x)\cdot y,
b\rangle, \cr &&\langle \sigma(L_{\cdot}(y)\otimes \id)\Delta(x),
a\otimes b\rangle = \langle x, (L_{\cdot}^*(y)b)\circ a\rangle =
\langle y\cdot(R_{\circ}^*(a)x), b\rangle, \cr &&\langle
(L_{\cdot}(y)\otimes \id)\Delta(x), a\otimes b\rangle = \langle x,
(L_{\cdot}^*(y)a)\circ b\rangle = \langle
L_{\circ}^*(L_{\cdot}^*(y)a)x,  b\rangle.
\end{eqnarray*}
Then Eq.~\eqref{eq_matched2} is equivalent to
Eq.~\eqref{eq_bialg2}. Hence the conclusion holds.
\end{proof}

\begin{rmk}\label{rmk:dual}
From the symmetry of the anti-flexible algebras $(A, \cdot)$ and $(A^*, \circ)$ in the
standard Manin triple of anti-flexible algebras associated to
$\mathfrak{B}_d$, we also can consider a linear map
$\gamma: A^*\rightarrow A^*\otimes A^*$ such that $\gamma^*:A\otimes A\rightarrow A$ gives the anti-flexible algebra structure ``$\cdot$" on $A$.
It is straightforward to show that $\Delta$ satisfies Eqs.~\eqref{eq_bialg1} and \eqref{eq_bialg2}
if and only if $\gamma$ satisfies
\begin{equation}\label{eq_dual_bialg1}
\gamma( a\circ b)+\sigma\gamma(b\circ a)=(\sigma(\id \otimes  L_{\circ}(b))+R_{\circ}(b)\otimes \id)\gamma (a)+
(\sigma( R_{\circ}(a)\otimes \id)+\id\otimes L_{\circ}(a))\gamma (b),
\end{equation}
\begin{equation}\label{eq_dual_bialg2}
\begin{array}{lll}
(\sigma(\id \otimes R_{\circ}(b))-\id\otimes R_{\circ}(b)-
\sigma(L_{\circ}(b)\otimes\id)
+(L_{\circ}(b)\otimes\id))\gamma(a)=\cr ((L_{\circ}(a)\otimes\id)-\sigma(L_{\circ}(a)\otimes\id)
+\sigma(\id\otimes R_{\circ}(a))
-(\id\otimes R_{\circ}(a)))\gamma(b),
\end{array}
\end{equation}
for any $a,b\in A^*$.
\end{rmk}

\begin{defi}
Let $(A,\cdot)$ be an anti-flexible algebra. An {\bf  anti-flexible bialgebra} structure on $A$ is a linear map
$\Delta:A\rightarrow A\otimes A$ such that
\begin{enumerate}
\item $\Delta^*:A^*\otimes A^*\rightarrow A^*$ defines an anti-flexible
algebra structure on $A^*$;
\item $\Delta$ satisfies Eqs.~\eqref{eq_bialg1} and \eqref{eq_bialg2}.
\end{enumerate}
We denote it by $(A,\Delta)$ or $(A,A^*)$.
\end{defi}

\begin{ex} \label{ex:dual} Let $(A,\Delta)$ be an anti-flexible bialgebra on an anti-flexible algebra $A$. Then $(A^*,\gamma)$ is an
anti-flexible bialgebra on the anti-flexible algebra $A^*$, where
$\gamma$ is given in Remark~\ref{rmk:dual}.
\end{ex}

Combining Proposition~\ref{Propo} and Theorem~\ref{thm:bialgebra}
together, we have the following conclusion.

\begin{thm}
Let $(A, \cdot)$ be an anti-flexible algebra. Suppose that there is an anti-flexible algebra structure on its dual space
$A^*$ denoted ``$\circ$" which is defined by a linear map $\Delta:A\rightarrow A\otimes A$.
Then
the following conditions are equivalent.
\begin{enumerate}
\item $(A\oplus A^*, A, A^*)$ is a standard Manin triple of
anti-flexible algebras associated to $\mathfrak{B}_d$ defined by
Eq.~\eqref{eq:sbl}. \item $(A, A^*, R_{\cdot}^*, L_{\cdot}^*,
R_{\circ}^*, L_{\circ}^*)$ is a matched pair of anti-flexible
algebras. \item $(A, \Delta)$ is an anti-flexible bialgebra.
\end{enumerate}
\end{thm}

Recall a Lie bialgebra structure on a Lie algebra $\frak g$ is a linear map $\delta:\frak g\rightarrow \frak g\otimes
\frak g$ such that $\delta^*:\frak g^*\otimes \frak g^*\rightarrow \frak g^*$ defines a Lie algebra structure on $\frak g^*$ and
$\delta$ satisfies
\begin{equation}
\delta[x,y]=({\rm ad}(x)\otimes {\rm id}+{\rm id}\otimes {\rm ad}(x))\delta(y)-({\rm ad}(y)\otimes {\rm id}+{\rm id}\otimes {\rm ad}(y))\delta(x),\;\;
\forall x,y\in \frak g,
\end{equation}
where ${\rm ad}(x)(y)=[x,y]$ for any $x,y\in\frak g$. We denoted it by $(\frak g,\delta)$.

\begin{pro}
Let $(A, \Delta)$ be an anti-flexible bialgebra.  Then $(\frak g(A),\delta)$ is a Lie bialgebra, where $\delta=\Delta-\sigma\Delta$.
\end{pro}

\begin{proof}
It is straightforward.
\end{proof}

\section{A special class of anti-flexible bialgebras}
% % % % % % % % % % % % % % % % % % % % % % % % % % % % %
% % % % % % % % % % % % % % % % % % % % % % % % % % % % % %

In this section, we consider a special class of anti-flexible bialgebras, 
that is, the anti-flexible bialgebra $(A,\Delta)$ on an anti-flexible algebra $(A,\cdot)$, with the linear map $\Delta$ defined by
\begin{equation}
\label{eq_coboundary}
\Delta(x)=(\id\otimes L_{\cdot}(x))\mathrm{r}+ 
(R_{\cdot}(x)\otimes\id)\sigma \mathrm{r},\;\;\forall x\in A,
\end{equation}
where $\mathrm{r}\in A\otimes A$.

\begin{lem}\label{lem:sigma}
Let $(A, \cdot )$ be an anti-flexible algebra and $\mathrm{r}\in A\otimes A$. Let $\Delta:A\rightarrow A\otimes A$ be a linear map defined by Eq.~\eqref{eq_coboundary}.
Then
\begin{equation}\label{eq:sigma_coboundary}
\sigma \Delta(x)=(L_{\cdot}(x)\otimes\id)\sigma \mathrm{r}+ (\id\otimes R_{\cdot}(x))\mathrm{r},\;\;\forall x\in A.
\end{equation}
\end{lem}

\begin{proof}
It is straightforward.
\end{proof}

\begin{pro}\label{pro:sum}
Let $(A, \cdot )$ be an anti-flexible algebra and $\mathrm{r}\in A\otimes A$. Let
$\Delta:A\rightarrow A\otimes A$ be a linear map defined by
Eq.~\eqref{eq_coboundary}.
\begin{enumerate}
\item  Eq.~\eqref{eq_bialg1} holds if and only if
\begin{equation}\label{eq_coboundary1}
(L_{\cdot}(x)\otimes R_{\cdot}(y)+R_{\cdot}(x)\otimes L_{\cdot}(y))(\mathrm{r}+\sigma \mathrm{r})=0,\;\;\forall x,y\in A.
\end{equation}
\item  Eq.~\eqref{eq_bialg2} holds if and only if
\begin{equation}\label{eq_coboundary2}
(R_{\cdot}(x)\otimes R_{\cdot}(y)- R_{\cdot}(y)\otimes R_{\cdot}(x)+
L_{\cdot}(x)\otimes L_{\cdot}(y)-L_{\cdot}(y)\otimes L_{\cdot}(x))(\mathrm{r}+\sigma \mathrm{r})=0,\;\;\forall x,y\in A.
\end{equation}
\end{enumerate}

\end{pro}

\begin{proof}
(a) Let $x, y\in A$. By Lemma~\ref{lem:sigma}, we have
$$
\Delta(x\cdot y)+\sigma\Delta(y\cdot x)=(\id\otimes (L_{\cdot}(x\cdot y)+
R_{\cdot}(y\cdot x)))\mathrm{r} +((R_{\cdot}(x\cdot y)+L_{\cdot}(y\cdot x))\otimes\id)\sigma \mathrm{r}.
$$
By the definition of an anti-flexible algebra, we have
\begin{equation*}\label{eq_qqq1}
\Delta(x\cdot y)+\sigma\Delta(y\cdot x)
= (\id\otimes (L_{\cdot}(x)L_{\cdot}(y)+R_{\cdot}(x)R_{\cdot}(y)))\mathrm{r}
 +((R_{\cdot}(y)R_{\cdot}(x)+L_{\cdot}(y)L_{\cdot}(x))\otimes\id)\sigma \mathrm{r}.
\end{equation*}
Moreover, we have
\begin{eqnarray*}
\sigma (\id\otimes L_{\cdot}(y))\Delta(x)&=&\sigma(\id\otimes L_{\cdot}(y))(\id\otimes L_{\cdot}(x))\mathrm{r}
+\sigma(\id\otimes L_{\cdot}(y))(R_{\cdot}(x)\otimes \id )\sigma \mathrm{r}\\
&=&(L_{\cdot}(y)L_{\cdot}(x)\otimes\id)\sigma \mathrm{r} +(L_{\cdot}(y)\otimes\id)(\id\otimes R_{\cdot}(x))\mathrm{r},\\
(R_{\cdot}(y)\otimes\id)\Delta(x)&=&   (R_{\cdot}(y)\otimes\id)(\id\otimes L_{\cdot}(x))\mathrm{r}
+(R_{\cdot}(y)\otimes\id)(R_{\cdot}(x)\otimes \id)\sigma \mathrm{r}\\
&=&(R_{\cdot}(y)\otimes\id)(\id\otimes L_{\cdot}(x))\mathrm{r}+(R_{\cdot}(y)R_{\cdot}(x)\otimes\id)\sigma \mathrm{r},\\
\sigma(R_{\cdot}(x)\otimes\id)\Delta(y)&=& \sigma(R_{\cdot}(x)\otimes\id)(\id\otimes L_{\cdot} (y))\mathrm{r}
 +\sigma(R_{\cdot}(x)\otimes\id)(R_{\cdot}(y)\otimes \id)\sigma \mathrm{r}\\
&=&(\id\otimes R_{\cdot}(x))(L_{\cdot}(y)\otimes\id)\sigma \mathrm{r}+(\id\otimes R_{\cdot}(x)R_{\cdot}(y))\mathrm{r},\\
(\id\otimes L_{\cdot} (x))\Delta(y)&=& (\id\otimes L_{\cdot} (x))(\id\otimes L_{\cdot} (y))\mathrm{r}
+(\id\otimes L_{\cdot} (x))(R_{\cdot}(y)\otimes \id)\sigma \mathrm{r}\\
&=&(\id \otimes L_{\cdot}(x)L_{\cdot}(y))\mathrm{r}+(\id\otimes L_{\cdot}(x))(R_{\cdot}(y)\otimes\id)\sigma \mathrm{r}.
\end{eqnarray*}
Hence we have
\begin{eqnarray*}\label{eq_qq1}
(R_{\cdot}(y)\otimes\id+\sigma (\id\otimes L_{\cdot}(y))
)\Delta(x)+(\id\otimes L_{\cdot}
(x)+\sigma(R_{\cdot}(x)\otimes\id))\Delta(y)\cr -\Delta(x\cdot
y)-\sigma\Delta(y\cdot x)=( R_{\cdot}(y)\otimes L_{\cdot}(x)
+L_{\cdot}(y)\otimes R_{\cdot}(x) )(\mathrm{r}+\sigma \mathrm{r}).
\end{eqnarray*}
Therefore Eq.~\eqref{eq_bialg1} hold if and only if Eq.
\eqref{eq_coboundary1} holds.

(b) Let $x,y\in A$. Then we have
\begin{eqnarray*}
\sigma(\id\otimes R_{\cdot}(y))\Delta(x)&=& \sigma(\id\otimes R_{\cdot}(y))(\id \otimes L_{\cdot}(x))\mathrm{r}
+\sigma(\id\otimes R_{\cdot}(y))(R_{\cdot}(x)\otimes \id)\sigma \mathrm{r}\\
&=&(R_{\cdot}(y)L_{\cdot}(x)\otimes\id)\sigma \mathrm{r}+(R_{\cdot}(y)\otimes\id)(\id\otimes R_{\cdot}(x))\mathrm{r},\\
(\id\otimes R_{\cdot}(y))\Delta(x) &=&(\id\otimes R_{\cdot}(y))(\id \otimes L_{\cdot}(x))\mathrm{r} +
(\id\otimes R_{\cdot}(y))(R_{\cdot}(x)\otimes \id)\sigma \mathrm{r} \\
&=&(\id\otimes R_{\cdot}(y)L_{\cdot}(x))\mathrm{r}+(\id\otimes R_{\cdot}(y))(R_{\cdot}(x)\otimes \id)\sigma \mathrm{r},\\
\sigma(L_{\cdot}(y)\otimes \id )\Delta(x)&=&  \sigma(L_{\cdot}(y)\otimes \id )(\id \otimes L_{\cdot}(x))\mathrm{r}+
\sigma(L_{\cdot}(y)\otimes \id )(R_{\cdot}(x)\otimes \id)\sigma \mathrm{r}\\
&=&(\id\otimes L_{\cdot}(y))(L_{\cdot}(x)\otimes\id)\sigma \mathrm{r}+(\id\otimes L_{\cdot}(y)R_{\cdot}(x))\mathrm{r},\\
(L_{\cdot}(y)\otimes \id )\Delta(x)&=& (L_{\cdot}(y)\otimes \id )(\id \otimes L_{\cdot}(x))\mathrm{r}+
 (L_{\cdot}(y)\otimes \id )(R_{\cdot}(x)\otimes \id)\sigma \mathrm{r}\\
&=&(L_{\cdot}(y)\otimes\id)(\id\otimes L_{\cdot}(x))\mathrm{r}+  (L_{\cdot}(y)R_{\cdot}(x)\otimes\id )\sigma \mathrm{r}.
\end{eqnarray*}
Therefore we have
\begin{eqnarray*}\label{eq_qq2}
\begin{array}{lll}
&&(\sigma(\id\otimes R_{\cdot}(y))+(L_{\cdot}(y)\otimes \id)
-(\id\otimes R_{\cdot}(y))-\sigma(L_{\cdot}(y)\otimes \id )
)\Delta(x)\cr &&-(\sigma(\id\otimes
R_{\cdot}(x))+(L_{\cdot}(x)\otimes \id) -(\id\otimes
R_{\cdot}(x))-\sigma(L_{\cdot}(x)\otimes \id ) )\Delta(y)\cr
&&=(R_{\cdot}(y)\otimes R_{\cdot}(x)+L_{\cdot}(y)\otimes
L_{\cdot}(x)-R_{\cdot}(x)\otimes R_{\cdot}(y)-L_{\cdot}(x)\otimes
L_{\cdot}(y))(\mathrm{r}+\sigma \mathrm{r})\cr &&+(([L_{\cdot}(y),
R_{\cdot}(x)]-[L_{\cdot}(x), R_{\cdot}(y)])\otimes \id)\sigma \mathrm{r}+
(\id\otimes ([L_{\cdot}(x), R_{\cdot}(y)]-[L_{\cdot}(y),
R_{\cdot}(x)]   ))\mathrm{r}\cr &&=(R_{\cdot}(y)\otimes
R_{\cdot}(x)+L_{\cdot}(y)\otimes L_{\cdot}(x)-R_{\cdot}(x)\otimes
R_{\cdot}(y)-L_{\cdot}(x)\otimes L_{\cdot}(y))(\mathrm{r}+\sigma \mathrm{r}).
\end{array}
\end{eqnarray*}
Note that the last equal sign is due to the definition of an
anti-flexible algebra. Hence
 Eq.~\eqref{eq_bialg2} hold if and only if Eq.~\eqref{eq_coboundary2} holds.
\end{proof}

\begin{lem}\label{Lem_coprod}
Let $A$ be a vector space and $\Delta: A  \rightarrow A\otimes A$
be a linear map. Then the dual map $\Delta^* : A^*\otimes A^*
\rightarrow A^*$ defines an anti-flexible  algebra structure on
$A^*$ if and only if $E_{\Delta} = 0$, where
\begin{equation}
E_{\Delta}=
(\Delta\otimes \id)\Delta-(\id\otimes\Delta)\Delta
+((\sigma\Delta)\otimes\id) (\sigma\Delta)-(\id\otimes(\sigma\Delta)) (\sigma\Delta).
\end{equation}
\end{lem}

\begin{proof} Denote by ``$\circ$" the product on $A^*$ defined by $\Delta^*$, that is,
$$\langle a\circ b, x\rangle =\langle \Delta^*(a\otimes b), x\rangle=\langle a\otimes b, \Delta(x)\rangle \;\;\forall x\in A, a,b\in A^*.$$
Therefore, for all $a, b, c\in A^*$ and $x\in A$, we have
\begin{eqnarray*}
\langle (a,b,c), x\rangle &=&\langle (a\circ b)\circ c-a\circ (b\circ c), x\rangle=
\langle \left(\Delta^*(\Delta^*\otimes\id)-\Delta^*(\id\otimes \Delta^*)  \right)(a\otimes b\otimes c),x\rangle\\
&=&\langle \left((\Delta\otimes\id) \Delta-(\id\otimes \Delta)\Delta\right)(x), a\otimes b\otimes c\rangle;\\
\langle (c,b,a),x\rangle &=&\langle (c\circ b)\circ a-c\circ (b\circ a),x\rangle =
\langle \left(\Delta^*(\Delta^*\otimes\id)-\Delta^*(\id\otimes \Delta^*)  \right)(c\otimes b\otimes a),x\rangle \cr
&=&\langle \left(  (\Delta^* \sigma^*) ((\Delta^*\sigma^*)\otimes\id)-  
(\Delta^* \sigma^*)(\id \otimes(\Delta^* \sigma^*))\right)(a\otimes b\otimes c),x\rangle\cr
&=&\langle \left(((\sigma\Delta)\otimes\id)(\sigma\Delta)-
(\id\otimes (\sigma\Delta))(\sigma\Delta)  \right)(x),a\otimes b\otimes c\rangle.
\end{eqnarray*}
Therefore, $(A^*,\circ)$ is an anti-flexible algebra if and only if $E_{\Delta}=0$.
\end{proof}

Let $(A,\cdot)$ be an anti-flexible algebra and 
$\displaystyle \mathrm{r}=\sum_i{a_i\otimes b_i}\in A\otimes A$. Set
\begin{equation}
\mathrm{r}_{12}=\sum_ia_i\otimes b_i\otimes 1,\quad
\mathrm{r}_{13}=\sum_{i}a_i\otimes 1\otimes b_i,\quad r_{23}=\sum_i1\otimes
a_i\otimes b_i,
\end{equation}
\begin{equation}
\mathrm{r}_{21}=\sum_ib_i\otimes a_i\otimes 1,\quad
\mathrm{r}_{31}=\sum_{i}b_i\otimes 1\otimes a_i,\quad r_{32}=\sum_i1\otimes
b_i\otimes a_i,
\end{equation}
where $1$ is the unit if $(A,\cdot)$ has a unit, otherwise is a
symbol playing a similar role of the unit. Then the operation
between two $\mathrm{r}$s is in an obvious way. For example,
\begin{equation}
\mathrm{r}_{12}\mathrm{r}_{13}=\sum_{i,j}a_i\cdot a_j\otimes b_i\otimes b_j,\;
\mathrm{r}_{13}\mathrm{r}_{23}=\sum_{i,j}a_i\otimes a_j\otimes b_i\cdot b_j,\;
\mathrm{r}_{23}\mathrm{r}_{12}=\sum_{i,j}a_j\otimes a_i\cdot b_j\otimes b_i,
\end{equation}
and so on.

\begin{thm}\label{thm:dual}
Let $(A, \cdot )$ be an anti-flexible algebra and $\mathrm{r}\in A\otimes A$. Let $\Delta:A\rightarrow A\otimes A$ 
be a linear map defined by Eq.~\eqref{eq_coboundary}.
Then $\Delta^*$ defines an anti-flexible algebra structure on $A^*$ if and only if
for any $x\in A$,
\begin{eqnarray}\label{eq_coboundarycoal}
&&(\id\otimes\id\otimes L_{\cdot}(x))(M(\mathrm{r}))+(\id\otimes\id\otimes R_{\cdot}(x))(P(\mathrm{r}))\nonumber\\
&&+( L_{\cdot}(x)\otimes\id\otimes \id)(N(\mathrm{r}))+( R_{\cdot}(x)\otimes\id\otimes\id)(Q(\mathrm{r}))=0,
\end{eqnarray}
where
\begin{eqnarray*}
&&M(\mathrm{r})=\mathrm{r}_{{23}}\mathrm{r}_{{12}}+\mathrm{r}_{{21}}r_{{13}}-
\mathrm{r}_{{13}}\mathrm{r}_{{23}},\;\;N(\mathrm{r})=\mathrm{r}_{{31}}\mathrm{r}_{{21}}-
\mathrm{r}_{{21}}\mathrm{r}_{{32}}-\mathrm{r}_{{23}}\mathrm{r}_{{31}},\\
&&P(\mathrm{r})=\mathrm{r}_{{13}}\mathrm{r}_{{21}}+\mathrm{r}_{{12}}r_{{23}}-
\mathrm{r}_{{23}}\mathrm{r}_{{13}},\;\; Q(\mathrm{r})=\mathrm{r}_{{21}}\mathrm{r}_{{31}}-
\mathrm{r}_{{31}}\mathrm{r}_{{23}}-\mathrm{r}_{{32}}\mathrm{r}_{{21}}.
\end{eqnarray*}
\end{thm}

\begin{proof}
Set
$\displaystyle \mathrm{r}=\sum_i a_i\otimes b_i$. Let $x\in A$. Then we
have
\begin{eqnarray*}
(\Delta\otimes\id) \Delta(x)&=&
\sum_{i, j} \{  a_j\otimes (a_i\cdot b_j)\otimes (x\cdot b_{i}) +
(b_j\cdot a_i)\otimes a_j\otimes (x\cdot b_i)\cr &+& a_j\otimes ((b_i\cdot x)b_j)\otimes a_i  + 
(b_j\cdot(b_i\cdot x))\otimes a_j\otimes a_i    \}\\
&=& \sum_{i, j} \{ a_j\otimes ((b_i\cdot x)\cdot b_j)\otimes a_i  + (b_j\cdot (b_i\cdot x))\otimes a_j\otimes a_i   \}
\cr&+&(\id\otimes\id\otimes L_{\cdot}(x))(\mathrm{r}_{{23}}\mathrm{r}_{{12}}+\mathrm{r}_{{21}}\mathrm{r}_{{13}}),\\
 ((\sigma\Delta)\otimes\id)(\sigma\Delta)(x)
&=&\sum_{i, j} \{ ((x\cdot b_i)\cdot b_j)\otimes a_j\otimes a_i+a_j\otimes (b_j\cdot (x\cdot b_i))\otimes a_i\cr
    &+&(a_i\cdot b_j)\otimes a_j\otimes (b_i\cdot x) +a_j\otimes(b_j\cdot a_i)\otimes(b_i\cdot x)        \}  \cr
    &=&\sum_{i, j} \{ ((x\cdot b_i)\cdot b_j)\otimes a_j\otimes a_i+a_j\otimes (b_j\cdot (x\cdot b_i))\otimes a_i\}
    \cr
    &+&(\id\otimes\id\otimes R_{\cdot}(x))(\mathrm{r}_{{13}}\mathrm{r}_{{21}}+\mathrm{r}_{{12}}\mathrm{r}_{{23}}),\\
(\id\otimes\Delta)\Delta(x)&=&\sum_{i, j} \{ a_i\otimes a_j\otimes((x\cdot b_i)\cdot b_j) +a_i\otimes(b_j\cdot (x\cdot b_i))\otimes a_j\cr
    &+&(b_i\cdot x)\otimes a_j\otimes (a_i\cdot b_j)
    +(b_i\cdot x)\otimes (b_j\cdot a_i)\otimes a_j  \}  \cr
    &=&\sum_{i, j} \{ a_i\otimes a_j\otimes((x\cdot b_i)\cdot b_j) +a_i\otimes(b_j\cdot (x\cdot b_i))\otimes a_j\}
    \cr &+& (R_{\cdot}(x)\otimes\id \otimes\id)(\mathrm{r}_{{31}}\mathrm{r}_{{23}}+\mathrm{r}_{{32}}\mathrm{r}_{{21}}),\\
(\id\otimes(\sigma\Delta))(\sigma\Delta)(x)&=&\sum_{i, j}   
\{ (x\cdot b_i)\otimes(a_i\cdot b_j)\otimes a_j+(x\cdot b_i)\otimes a_j \otimes (b_j\cdot a_i)\cr
    &+&
    a_i\otimes ((b_i\cdot x)\cdot b_j)\otimes a_j + a_i\otimes a_j\otimes (b_j\cdot (b_i\cdot x)) \}
    \cr
    &=& \sum_{i, j}   \{ a_i\otimes ((b_i\cdot x)\cdot b_j)\otimes a_j + a_i\otimes a_j\otimes (b_j\cdot (b_i\cdot x))\}
    \cr &+& (L_{\cdot}(x)\otimes\id \otimes\id)(\mathrm{r}_{{21}}\mathrm{r}_{{32}}+\mathrm{r}_{{23}}\mathrm{r}_{{31}}).
\end{eqnarray*}
Thus $E_{\Delta}(x)=(A1)+(A2)+(A3)$, where
\begin{eqnarray*}
(A1)&=&
(\id\otimes\id\otimes L_{\cdot}(x))(\mathrm{r}_{{23}}\mathrm{r}_{{12}}+\mathrm{r}_{{21}}\mathrm{r}_{{13}})+
(\id\otimes\id\otimes R_{\cdot}(x))(\mathrm{r}_{{13}}\mathrm{r}_{{21}}+\mathrm{r}_{{12}}\mathrm{r}_{{23}})\cr
&-&(R_{\cdot}(x)\otimes\id \otimes\id)(\mathrm{r}_{{31}}\mathrm{r}_{{23}}+\mathrm{r}_{{32}}\mathrm{r}_{{21}})
-(L_{\cdot}(x)\otimes\id \otimes\id)(\mathrm{r}_{{21}}\mathrm{r}_{{32}}+\mathrm{r}_{{23}}\mathrm{r}_{{31}}), \cr
(A2) &=&\sum_{i, j} \{ a_j\otimes ((b_i\cdot x)\cdot b_j+b_j\cdot (x\cdot b_i))\otimes a_i- 
a_i\otimes(b_j\cdot (x\cdot b_i)+(b_i\cdot x)\cdot b_j)\otimes a_j   \},\cr
(A3)&=&\sum_{i, j} \{ ((x\cdot b_i)\cdot b_j+b_j\cdot (b_i\cdot x))\otimes a_j\otimes a_i -
a_i\otimes a_j\otimes((x\cdot b_i)\cdot b_j+b_j\cdot (b_i\cdot x)) \}.
\end{eqnarray*}
By exchanging the indices $i$ and $j$, we have $(A2)=0$.

\noindent By the definition of an anti-flexible algebra,  we have
\begin{eqnarray*}
(A3)&=&\sum_{i, j} \{ (x\cdot (b_i\cdot b_j)+(b_j\cdot b_i)\cdot x)\otimes a_j\otimes a_i -
a_i\otimes a_j\otimes(x\cdot (b_i\cdot b_j)+(b_j\cdot b_i)\cdot x) \}\\
&=&(L_{\cdot}(x)\otimes\id\otimes \id)(\mathrm{r}_{31}\mathrm{r}_{21})+
(R_{\cdot}(x)\otimes\id \otimes\id)(\mathrm{r}_{21}\mathrm{r}_{31})
\cr&-&
(\id\otimes\id\otimes R_{\cdot}(x))(\mathrm{r}_{23}\mathrm{r}_{13})-
(\id\otimes\id \otimes L_{\cdot}(x))(\mathrm{r}_{13}\mathrm{r}_{23}).
\end{eqnarray*}
Then we  have 
\begin{eqnarray*}
E_{\Delta}(x)&=&(\id\otimes\id\otimes
L_{\cdot}(x))(\mathrm{r}_{{23}}\mathrm{r}_{{12}}+
\mathrm{r}_{{21}}\mathrm{r}_{{13}}-\mathrm{r}_{{13}}\mathrm{r}_{{23}})+
(\id\otimes\id\otimes
R_{\cdot}(x))(\mathrm{r}_{{13}}\mathrm{r}_{{21}}+
\mathrm{r}_{{12}}\mathrm{r}_{{23}}-\mathrm{r}_{{23}}\mathrm{r}_{{13}})\\
&-&(R_{\cdot}(x)\otimes\id
\otimes\id)(\mathrm{r}_{{31}}\mathrm{r}_{{23}}+\mathrm{r}_{{32}}\mathrm{r}_{{21}}
-\mathrm{r}_{{21}}\mathrm{r}_{{31}})-(L_{\cdot}(x)\otimes\id
\otimes\id)(\mathrm{r}_{{21}}\mathrm{r}_{{32}}+\mathrm{r}_{{23}}\mathrm{r}_{{31}}-
\mathrm{r}_{{31}}\mathrm{r}_{{21}}).
\end{eqnarray*}
Hence the conclusion follows.
\end{proof}

\begin{rmk} \label{rmk:M}
In fact, for any $\mathrm{r}\in A\otimes A$, we have
$$N(\mathrm{r})=-\sigma_{13}M(\mathrm{r}),\;\;\; P(\mathrm{r})=\sigma_{12}M(\mathrm{r}), \;\;\;
Q(\mathrm{r})=-\sigma_{12}\sigma_{13}M(\mathrm{r}),$$ where $\sigma_{12}(x\otimes
y\otimes z)=y\otimes x\otimes z$, $\sigma_{13}(x\otimes y\otimes
z)=z\otimes y\otimes x$, for any $x,y,z\in A$.
\end{rmk}
Combing Proposition~\ref{pro:sum}, Theorem~\ref{thm:dual} and
Remark~\ref{rmk:M} together, we have the following conclusion.

\begin{thm}\label{thm:main}
Let $(A, \cdot )$ be an anti-flexible algebra and $\mathrm{r}\in A\otimes A$. Let $\Delta:A\rightarrow A\otimes A$
be a linear map defined by Eq.~\eqref{eq_coboundary}.
Then $(A, \Delta)$ is an anti-flexible bialgebra if and only if $r$ satisfies Eqs.~\eqref{eq_coboundary1},
\eqref{eq_coboundary2} and
\begin{equation}\label{eq_coboundarycoal1}
\begin{array}{lll}
\left((\id\otimes\id\otimes L_{\cdot}(x))- (
R_{\cdot}(x)\otimes\id\otimes\id)\sigma_{12}\sigma_{13}
\right.\cr +\left.((\id\otimes\id\otimes R_{\cdot}(x))\sigma_{12}-
( L_{\cdot}(x)\otimes\id\otimes \id ) \sigma_{13}\right   )M(\mathrm{r})=0,
\end{array}
\end{equation}
where $M(\mathrm{r})=\mathrm{r}_{{23}}\mathrm{r}_{{12}}+
\mathrm{r}_{{21}}\mathrm{r}_{{13}}-\mathrm{r}_{{13}}\mathrm{r}_{{23}}$.
\end{thm}

As a direct consequence of Theorem~\ref{thm:main}, we have the
following result.

\begin{cor}
Let $(A, \cdot )$ be an anti-flexible algebra and $\mathrm{r}\in A\otimes A$. Let $\Delta:A\rightarrow A\otimes A$
be a linear map defined by Eq.~\eqref{eq_coboundary}.
If in addition, $\mathrm{r}$ is  skew-symmetric and $\mathrm{r}$ satisfies
\begin{equation}\label{eq_YBE}
\mathrm{r}_{{12}}\mathrm{r}_{{13}}-\mathrm{r}_{{23}}\mathrm{r}_{{12}}+
\mathrm{r}_{{13}}\mathrm{r}_{{23}}=0,
\end{equation}
then $(A, \Delta)$ is an anti-flexible bialgebra.
\end{cor}

\begin{rmk}
In fact, there is certain "freedom degree" for the construction of $\Delta:A\rightarrow A\otimes A$
defined by Eq.~\eqref{eq_coboundary}. Explicitly, assume
\begin{equation}\label{eq:Delta'}
\Delta'(x)=\Delta(x)+\Pi(x)(\mathrm{r}+\sigma (\mathrm{r}))= 
({\rm id}\otimes L_\cdot(x)+\Pi(x))\mathrm{r}+(R_\cdot(x)\otimes {\rm id}+
\Pi(x))\sigma(\mathrm{r}),\;\;\forall x\in A,
\end{equation}
where $\Pi(x)$ is an operator depending on $x$ acting on $A\otimes A$.
 Then by a direct and similar proof as of Theorem~\ref{thm:main} or by Theorem~\ref{thm:main} 
 through the relationship between $\Delta'$ and $\Delta$, one can show
that $(A,\Delta')$ is an anti-flexible bialgebra if and only if the following equations hold:
\begin{eqnarray*}
&&{\rm LHS}\;\;{\rm of} \;\;{\rm Eq.}\;\; (\ref{eq_coboundary1})+A(x)(\mathrm{r}+\sigma(\mathrm{r}))=0,\\
&&{\rm LHS}\;\;{\rm of} \;\;{\rm Eq.}\;\; (\ref{eq_coboundary2})+B(x)(\mathrm{r}+\sigma(\mathrm{r}))=0,\\
&&{\rm LHS}\;\;{\rm of} \;\;{\rm Eq.}\;\; (\ref{eq_coboundarycoal1})+C_{12}(x)(\mathrm{r}_{12}+\mathrm{r}_{21})
+C_{23}(x)(\mathrm{r}_{23}+\mathrm{r}_{32})+C_{13}(x)(\mathrm{r}_{13}+\mathrm{r}_{31})=0,
\end{eqnarray*}
where $A(x), B(x)$ are operators depending on $x$ acting on $A\otimes A$, $C_{12}(x),C_{23}(x),C_{13}(x)$ are operators
depending on $x$ acting on $A\otimes A\otimes A$ (it is the component itself when the component acts on $1$), and all of them are related to $\Pi(x)$.
Hence by this conclusion  (which is independent of Theorem~\ref{thm:main}) directly,  we still show that $(A,\Delta')$ 
is an anti-flexible bialgebra when $\mathrm{r}$ is skew-symmetric and $\mathrm{r}$ satisfies Eq.~(\ref{eq_YBE}).
That is, this $\Delta'$ defined by Eq.~(\ref{eq:Delta'}), also leads to the introduction of Eq.~(\ref{eq_YBE}). 
Note that when $\mathrm{r}$ is skew-symmetric, $\Delta'=\Delta$.
\end{rmk}

\begin{defi}
Let $(A, \cdot )$ be an anti-flexible algebra and $\mathrm{r}\in A\otimes A$. Eq.~(\ref{eq_YBE}) is called the
{\bf anti-flexible Yang-Baxter equation} (AFYBE) in $(A,\cdot)$.
\end{defi}

\begin{rmk}
The notion of anti-flexible Yang-Baxter equation in an anti-flexible algebra is due
to the fact that it is an analogue of the classical Yang-Baxter equation
in a Lie algebra (\cite{Drinfeld}) or the associative Yang-Baxter
equation in an associative algebra (\cite{Bai_Double}).
\end{rmk}

It is a remarkable observation and an unexpected consequence that both the anti-flexible
Yang-Baxter equation in an anti-flexible algebra and the associative Yang-Baxter equation (\cite{Bai_Double}) in an associative algebra
 have the same form Eq.~(\ref{eq_YBE}). Hence both these two equations have some common properties. At the end of this section, we give two properties of anti-flexible Yang-Baxter equation whose proofs are omitted since
the proofs are the same as in the case of associative Yang-Baxter equation.

Let $A$ be a vector space. For any $\mathrm{r}\in
A\otimes A$, $\mathrm{r}$ can be regarded as a linear map from $A^*$ to $A$ in the
following way:
\begin{equation}\langle \mathrm{r}, u^*\otimes v^*\rangle   =\langle  \mathrm{r}(u^*),v^*\rangle  
 ,\;\;\forall\; u^*,v^*\in A^*.
 \end{equation}

\begin{pro}\label{pro:of}
    Let $(A, \cdot)$ be an anti-flexible algebra and $\mathrm{r}\in A\otimes A$ be skew-symmetric.
Then $\mathrm{r}$ is solution of anti-flexible Yang-Baxter equation if and only
    if $\mathrm{r}$ satisfies
    \begin{equation}\label{eq:of}
    \mathrm{r}(a)\cdot \mathrm{r}(b)=\mathrm{r}(R_{\cdot}^*(\mathrm{r}(a))b+
    L_{\cdot}^*(\mathrm{r}(b))a),\;\;\forall a,b\in A^*.
    \end{equation}
\end{pro}

\begin{rmk}
Since the dual bimodules of both anti-flexible and associative
algebras have the same form (see Remark~\ref{rmk:same}), the
interpretation of anti-flexible Yang-Baxter equation in terms of
operator form ~(\ref{eq:of}) in the above Proposition~\ref{pro:of}
explains partly why the anti-flexible Yang-Baxter equation has the
same form as  of the associative Yang-Baxter equation.
\end{rmk}

\begin{thm}  Let $(A, \cdot)$ be an anti-flexible algebra and $\mathrm{r}\in A\otimes A$.
 Suppose that $\mathrm{r}$ is antisymmetric and
nondegenerate. Then $\mathrm{r}$ is a solution of anti-flexible Yang-Baxter
equation in $(A,\cdot)$ if and only if the inverse of the isomorphism
$A^*\rightarrow A$ induced by $\mathrm{r}$, regarded as a bilinear form
$\omega$ on $A$ (that is, $\omega(x,y)=\langle  \mathrm{r}^{-1}x,y\rangle$
for any $x,y\in A$), satisfies
\begin{equation}\label{eq:Conn}
    \omega(x\cdot y, z)+\omega(y\cdot z, x)+\omega(z\cdot x, y)=0,\;\;\forall x,y,z\in A.
    \end{equation}
\end{thm}

\section{$\mathcal O$-operators of anti-flexible algebras and pre-anti-flexible algebras}

In this section, we introduce the notions of $\mathcal O$-operators of 
anti-flexible algebras and pre-anti-flexible algebras to construct
skew-symmetric solutions of anti-flexible Yang-Baxter equation and hence to construct
anti-flexible bialgebras.

\begin{defi}
Let $(l,r, V)$ be a bimodule of  an anti-flexible algebra $(A, \cdot)$.
A linear map $T:V\rightarrow A$ is called an {\bf $\mathcal{O}$-operator associated to
$(l,r, V)$} if $T$ satisfies
\begin{equation}
T(u)\cdot T(v)=T(l(T(u))v+r(T(v))u),\;\;\forall u,v\in V.
\end{equation}
\end{defi}

\begin{ex}
Let $(A,\cdot)$ be an anti-flexible algebra. An $\mathcal O$-operator $R_B$ associated to the regular bimodule $(L,R,A)$ is
called a {\bf Rota-Baxter operator of weight zero}, that is, $R_B$
satisfies
\begin{equation}
R_B(x)\cdot R_B(y)=R_B(R_B(x)\cdot y+x\cdot R_B(y)),\;\;\forall
x,y\in A.
\end{equation}
\end{ex}

\begin{ex}
Let $(A,\cdot)$ be an anti-flexible algebra and $\mathrm{r}\in A\otimes A$.
If $\mathrm{r}$ is skew-symmetric, then by Proposition~\ref{pro:of}, $\mathrm{r}$ is
a solution of anti-flexible Yang-Baxter equation if and only if
$\mathrm{r}$ regarded as a linear map from $A^*$ to $A$ is an $\mathcal
O$-operator associated to the bimodule
$(R^*_\cdot,L^*_\cdot,A^*)$.
\end{ex}

There is the following construction of (skew-symmetric) solutions
of anti-flexible Yang-Baxter equation in a semi-direct product
anti-flexible algebra from an $\mathcal O$-operator of an
anti-flexible algebra which is similar as for associative algebras
(\cite[Theorem 2.5.5]{Bai_Double}, hence the proof is omitted).

\begin{thm}\label{thm:cfromO}
Let $(l,r, V)$ be a bimodule of  an anti-flexible algebra $(A, \cdot)$.
Let $T: V\rightarrow A$ be a linear map which is identified as an element in $(A\ltimes_{r^*, l^*} V^* )\oplus (A\ltimes_{r^*, l^*} V^*)$.
Then $\mathrm{r}=T-\sigma (T)$ is a skew-symmetric solution of anti-flexible Yang-Baxter equation in $A\ltimes_{r^*, l^*} V^*$ if only if $T$ is an
    $\mathcal{O}$-operator associated to the bimodule $(l, r, V)$.
\end{thm}

\begin{defi}
Let $A$ be a vector space with two bilinear products $\prec, \succ: A\otimes A \rightarrow A$.
We call it a {\bf pre-anti-flexible algebra} denoted by $(A, \prec, \succ)$ if for any $x,y,z\in A$, the following
equations are satisfied
 \begin{equation}
(x,y,z)_{_m}=(z,y,x){_m},
\end{equation}
\begin{equation}
(x,y,z)_{_l}=(z,y,x)_{_r},
\end{equation}
where
\begin{equation}\label{eq_dendri_m}
(x,y,z)_{_m}:=(x \succ y) \prec z-x \succ (y \prec z),
\end{equation}
\begin{equation}\label{eq_dendri_l}
(x,y,z)_{_l}:=(x\ast y)\succ z-x\succ (y\succ z),
\end{equation}
\begin{equation}\label{eq_dendri_r}
(x,y,z)_{_r}:=(x\prec y)\prec z-x\prec (y\ast z),
\end{equation}
here $x\ast y=x\prec y+x\succ y$.
\end{defi}

\begin{rmk}
Note that if both hand sides in Eqs.~\eqref{eq_dendri_m}, \eqref{eq_dendri_l} and \eqref{eq_dendri_r} are zero, that is,
\begin{equation}
(x,y,z)_{_m}=0, \quad (z,y,x)_{_l}=0, \quad (x,y,z)_{_r}=0,
\end{equation}
then it exactly gives the definition of a {\bf dendriform algebra} which was introduced by Loday in \cite{Loday}. Hence
any dendriform algebra is a pre-anti-flexible algebra, that is, pre-anti-flexible algebras can be regarded as a natural generalization of dendriform algebras.
On the other hand, from the point of view of operads, like dendriform algebras being the splitting of associative algebras,
pre-anti-flexible algebras are the splitting of anti-flexible algebras (\cite{BBGN,PBG}).
\end{rmk}

\begin{pro}\label{pro:asso}
Let $(A, \prec,\succ)$ be a pre-anti-flexible algebra. Define a bilinear product $\ast:A\otimes A\rightarrow A$ by
\begin{equation}\label{eq_pre_anti}
    x\ast y= x\prec y+ x\succ y,\;\;\forall x,y\in A.
    \end{equation}
Then $(A,\ast)$ is an anti-flexible algebra, which is called the {\bf associated anti-flexible algebra of $(A,\prec,\succ)$}.
\end{pro}

\begin{proof}
 Set
$$(x,y,z)_\ast=(x\ast y)\ast z-x\ast (y\ast z),\;\;\forall x,y,z\in A.$$
Then for any $x,y,z\in A$, we have
$$
(x,y,z)_\ast=(x,y,z)_{_m}+(x,y,z)_{_l}+(x,y,z)_{_r}
=
(z,y,x)_{_m}+(z,y,x)_{_l}+(z,y,x)_{_r}=(z, y, x)_\ast.$$
Hence $(A,\ast)$ is an anti-flexible algebra.
\end{proof}

Let $(A, \prec, \succ)$ be a pre-anti-flexible algebra. For any $x\in A$,
let $L_\succ(x), R_\prec(x)$ denote
the left multiplication operator of $(A,\prec)$ and the right multiplication operator of
$(A,\succ)$ respectively, that is, $L_\succ (x)(y)=x\succ y,
\;\;
R_\prec(x)(y)=y\prec x,\;\;\forall\;x, y\in A$. Moreover, let
$L_\succ, R_\prec: A\rightarrow \frak{gl}(A)$ be two
linear maps with $x\rightarrow L_\succ(x)$ and $x\rightarrow R_\prec (x)$
respectively.

\begin{pro}
 Let $(A, \prec, \succ)$ be a pre-anti-flexible algebra. Then  $(L_{_\succ}, R_{_\prec}, A)$
 is a bimodule of the associated anti-flexible algebra $(A,\ast)$, where $\ast$ is defined by Eq.~\eqref{eq_pre_anti}.
\end{pro}

\begin{proof}
For any $x,y,z\in A$, we have
\begin{eqnarray*}
&&(L_\succ (x\ast y)-L_\succ(x)L_\succ(y))(z)
=(x\ast y) \succ z- x \succ (y \succ z)=(x,y,z)_{_l},\\
&&(R_\prec(x)R_\prec(y)-R_\prec(y\ast x))(z)
(z\prec y)\prec x-z\prec (y\ast x)=(z,y,x)_r,\\
&&(L_\succ(x)R_\prec(y)-R_\prec(y)L_\succ(x))(z)=x\succ(z\prec y)-(x\succ z )\prec y
=(x,z,y)_m,\\
&&(L_\succ(y)R_\prec(x)-R_\prec(x)L_\succ(y))(z)=y\succ(z\prec x)-(y\succ z )\prec x
=(y,z,x)_m.
\end{eqnarray*}
Hence $(L_{_\succ}, R_{_\prec}, A)$ is a bimodule of $(A,\ast)$.
\end{proof}

A direct consequence is given as follows.

\begin{cor} \label{cor:id}
Let $(A, \prec, \succ)$ be a pre-anti-flexible algebra. Then the identity map ${\rm id}$ is an $\mathcal O$-operator of the associated
anti-flexible algebra $(A,\ast)$ associated to the bimodule $(L_\succ,R_\prec,A)$.
\end{cor}

\begin{thm} Let $(l,r, V)$ be a bimodule of  an anti-flexible algebra $(A, \cdot)$. Let $T:V\rightarrow A$ be an ${\mathcal O}$-operator
associated to $(l,r,V)$. Then there exists a pre-anti-flexible algebra
structure on $V$ given by
\begin{equation}\label{eq:Vp}
u\succ v=l(T(u))v,\;\;u\prec v=r(T(v))u,\;\;\forall\; u,v\in
V.\end{equation} So there is an associated anti-flexible algebra
structure on $V$ given by Eq.~\eqref{eq_pre_anti} and $T$ is a
homomorphism of anti-flexible algebras. Moreover,
$T(V)=\{T(v)|v\in V\}\subset A$ is an anti-flexible subalgebra of
$(A,\cdot)$ and there is an induced pre-anti-flexible algebra
structure on $T(V)$ given by
\begin{equation}\label{eq:Ap}T(u)\succ T(v)=T(u\succ v),\;\;T(u)\prec T(v)=T(u\prec v),\;\;\forall\; u,v\in V.\end{equation}
Its corresponding associated anti-flexible algebra structure on
$T(V)$ given by Eq.~\eqref{eq_pre_anti} is just the anti-flexible
subalgebra structure of $(A,\cdot)$ and $T$ is a homomorphism of
pre-anti-flexible algebras.\end{thm}

\begin{proof}
For all $u,v,w\in V$, we have
{\small\begin{eqnarray*}
(u,v, w)_{_m}&=&(u\succ v)\prec w-u\succ (v\prec w)=r(T(w))l(T(u))v-l(T(u))r(T(w))v\cr
&=& r(T(u))l(T(w))v-l(T(u))r(T(w))v=(w,v,u)_{_m},\\
(u,v, w)_{_l}&=&(u\succ v+u\prec v)\succ w-u\succ (v\succ w)=(l(T( l(T(u))v +r(T(v))u))-l(T(u))l(T(v)))w\cr
&=& (l(T(u)\cdot T(v))-l(T(u))l(T(v)))w=(r(T(u))r(T(v))-r(T(v)\cdot T(u))w\\
& = & (r(T(u))r(T(v))-r(T(u\succ v+u\prec v)))w =(w\prec v)\prec
u-w\prec(u\succ v+u\prec v )\cr &=&(w,v,u){_r}
\end{eqnarray*}}
Therefore, $(V,\prec, \succ)$ is a pre-anti-flexible algebra. For $T(V)$, we have
$$T(u)*T(v)=T(u\succ v+u\prec v)=T(u*v)=T(u)\cdot T(v),\;\;\forall u,v\in V.$$
The rest is straightforward.
\end{proof}

\begin{cor}\label{proposition_existence}
Let $(A,\cdot)$ be an anti-flexible algebra. Then there exists
a pre-anti-flexible algebras structure on $A$ such that its associated
anti-flexible algebra is $(A,\cdot)$ if and only if there
exists an invertible $\mathcal{O}$-operator.
\end{cor}

\begin{proof}
Suppose that there exists an invertible $\mathcal{O}$-operator $T:V\rightarrow A$ associated to a bimodule $(l,r, V)$.
Then the products ``$\succ, \prec$" given by Eq.~(\ref{eq:Vp}) defines a pre-anti-flexible algebra structure on $V$. Moreover,
there is a pre-anti-flexible algebra structure on $T(V)=A$ given by Eq.~(\ref{eq:Ap}), that is,
$$x\succ y=T(l(x)T^{-1}(y)),\;\; x\prec y=T(r(y)T^{-1}(x)),\;\;\forall x,y\in A.$$
Moreover, for any $x,y\in A$, we have
$$x\succ y+x\prec y=T(l(x) T^{-1}(y)+r(y)T^{-1}(x))=T(T^{-1}(x))\cdot T(T^{-1}(y))=x\cdot y.$$
Hence the associated anti-flexible algebra of $(A,\succ,\prec)$ is $(A,\cdot)$.

Conversely, let $(A,\succ, \prec)$ be pre-anti-flexible algebra
such that its associated anti-flexible is $(A,\cdot)$. Then by
Corollary~\ref{cor:id}, the identity map ${\rm id}$ is an
$\mathcal O$-operator of $(A,\cdot)$ associated to the bimodule
$(L_\succ, R_\prec, A)$.
\end{proof}

\begin{cor}  Let $(A,\cdot)$ be an anti-flexible algebra and $\omega$ be a
nondegenerate skew-symmetric bilinear form satisfying
Eq.~\eqref{eq:Conn}. Then there exists a pre-anti-flexible algebra
structure $\succ,\prec$ on $A$ given by
\begin{equation}\label{eq:Conn-pre}\omega(x\succ y,z)=\omega(y, z\cdot x),\;\; \omega(x\prec
y,z)=\omega(x,y\cdot z),\;\; \forall x,y,z\in A,\end{equation}
such that the associated anti-flexible algebra is $(A,\cdot)$.
\end{cor}

\begin{proof} Define a linear map $T:A\rightarrow A^*$ by
$$\langle T(x),y\rangle=\omega (x,y),\;\;\forall x,y\in A.$$ Then $T$
is invertible and $T^{-1}$ is an $\mathcal O$-operator of the
anti-flexible  algebra $(A,\cdot)$ associated to the bimodule
$(R^*_{\cdot},L^*_{\cdot}, A^*)$. By Corollary~
\ref{proposition_existence}, there is a pre-anti-flexible
 algebra structure $\succ,\prec$ on $(A,*)$ given by
$$x\succ y=T^{-1}R^*(x)T(y),\;\;x\prec y=T^{-1}L^*(y)T(x),\;\;\forall
x,y\in A,$$ which gives exactly Eq.~(\ref{eq:Conn-pre}) such that the associated anti-flexible algebra is $(A,\cdot)$.
\end{proof}

Finally we give the following construction of skew-symmetric solutions of
anti-flexible Yang-Baxter equation (hence anti-flexible bialgebras) from a pre-anti-flexible algebra.

\begin{pro}
 Let $(A,\succ,\prec)$ be a pre-anti-flexible algebra.
Then
\begin{equation}
\mathrm{r}=\sum_{i}^n (e_i\otimes e_i^*-e_i^*\otimes e_i)
\end{equation}
is a solution of anti-flexible Yang-Baxter equation in $A
\ltimes_{R_\prec^*,L_\succ^*} A^*$, where $\{e_1,\cdots, e_n\}$ is a
basis of $A$ and $\{e_1^*,\cdots, e_n^*\}$ is its dual basis.
\end{pro}

\begin{proof} Note that the identity map ${\rm id}=\sum\limits_{i=1}^n e_i\otimes e_i^*$. Hence the conclusion follows from
Theorem~\ref{thm:cfromO} and Corollary~\ref{cor:id}.
\end{proof}

\bigskip

\noindent
{\bf Acknowledgements.}  This work is supported by
NSFC (11931009).  C. Bai is also
supported by the Fundamental Research Funds for the Central
Universities and Nankai ZhiDe Foundation.

\end{document}